%% file: exceptionalheart.tex
\documentclass[leqno,twoside]{article}
\input{structure-article}

\usepackage[autostyle]{csquotes} 
\usepackage[backend=biber, style=numeric]{biblatex}

\usepackage{hyperref}
\newcommand{\id}{\operatorname{id}}

\newcommand{\cone}{\operatorname{C}}
\newcommand{\im}{\operatorname{im}}
\newcommand{\der}{{\operatorname{D}^{\operatorname{b}}}}
\newcommand{\T}{\mathcal{T}}
\newcommand{\coker}{\operatorname{coker}}
\newcommand{\ext}{\operatorname{Ext}}
\newcommand{\rdim}{\operatorname{rdim}}
\renewcommand{\hom}{\operatorname{Hom}}
\newcommand{\gr}{\operatorname{gr}}
\begin{document}
\titlepage
\begin{center}
{\LARGE \textbf{Compatibility of t-structures in a semiorthogonal decomposition}}
\par\vspace{1ex}
\textit{Antonio Lorenzin}
\par \vspace{1ex}
{\footnotesize Dipartimento di Matematica "F. Casorati", Università degli Studi di Pavia, Via Ferrata 5, 27100 Pavia, Italy
\par\vspace{-0.5ex}
E-mail address: \texttt{a.lorenzin@campus.unimib.it}}

 \hfill Pavia, \today
\end{center}
\tableofcontents
\par\vspace{5ex}
\noindent {\large \textbf{Abstract}}\hspace{2ex} We describe how to obtain a global t-structure from a semiorthogonal decomposition with compatible t-structures on every component. 
This result is used to generalize a well-known theorem of Bondal on full strong exceptional sequences.
\par\vspace{2ex}

\noindent {\large \textbf{Keywords}}\hspace{2ex} triangulated categories;  t-structures;  semiorthogonal decomposition;  exceptional sequences;  filtered enhancement.
\par\vspace{2ex}

\noindent {\large \textbf{Acknowledgements}}\hspace{2ex} I am very thankful to my supervisor Alberto Canonaco for his suggestions and corrections. I would also like to thank  Irakli Patchkoria and Xiao-Wu Chen, who pointed out the existence of Hubery's result \cite[Theorem 3.2]{hubery} after the first preliminary version of this article. I am grateful to the referee who provided insightful comments.

\phantomsection
\addcontentsline{toc}{section}{Introduction}
\section*{Introduction}

The notion of algebraic triangulated categories has raised great attention in the last decades. These particular triangulated categories are described in several equivalent ways: they admit an enhancement by dg categories or $A_{\infty}$-categories; alternatively, they are obtained as the stable category of a Frobenius category.
The most important examples are derived categories (assuming they are categories).\footnote{Here, and in the following, category is used to mean \emph{locally small category}.}
Interestingly, many results can be extended from derived categories to algebraic triangulated categories. We focus on the following.
\par\vspace{1ex}\noindent 
\textbf{Theorem -- Bondal.} \cite[Theorem 6.2]{bondal} \textit{Assume that the bounded derived category $\der(X)$ of coherent sheaves on a smooth manifold $X$ is generated by a strong exceptional sequence $\langle E_1 , \dots , E_n \rangle$. Then $\der(X)$ is equivalent to the bounded derived category $\der(\operatorname{mod-}A)$ of right finite-dimensional modules over the algebra $A=\operatorname{End} (\bigoplus_{i=1}^n E_i)$.}

\par\vspace{1ex}
\noindent 
Bondal's result  has been generalized to algebraic triangulated categories by Keller (\cite[Theorem 8.7]{keller}). In particular, the statement below is a consequence of Keller's work.
\par\vspace{1ex}\noindent 
\textbf{Theorem -- Keller-Orlov.} \cite[Corollary 1.9]{orlov} \textit{Let $\T$ be an algebraic triangulated category. Assume that $\T$ has a full strong exceptional sequence $\langle E_1 , \dots, E_n \rangle$. Then the category $\T$ is equivalent to the derived category $\der(\operatorname{mod-}A)$, where $A=\operatorname{End} (\bigoplus_{i=1}^n E_i )$ is the algebra of endomorphisms of the collection $\langle E_1 , \dots, E_n \rangle$.}
\par\vspace{1ex}\noindent 
A question may arise: is it possible to drop the algebraic requirement? At the moment, no answer has been found. As a matter of fact, it is incredibly hard to study the general case of triangulated categories; indeed, the definition of an exceptional object requires the category to be $\mathbb{K}$-linear, with $\mathbb{K}$ a field, and the only known example of non-algebraic $\mathbb{K}$-linear triangulated category is studied in \cite{rizzvdb}.

Our aim is to generalize Bondal's result. For this reason, we deal with the construction of a global t-structure, starting with compatible t-structures on semiorthogonal components. Surprisingly, the result is not hard to prove and it follows from basic theory. As a corollary, a full strong exceptional sequence of length 2 gives a heart of dimension at most 1, so that Hubery's result \cite[Theorem 3.2]{hubery} can be applied without any additional requirement on the triangulated category. We obtain the following.

\par\vspace{1ex}\noindent 
\textbf{\ref{cor:uniex}. Corollary.} Let $\mathbb{K}$ be a field. Any $\mathbb{K}$-linear triangulated category $\T$ with a full strong exceptional sequence $\langle E_1 , E_2 \rangle $ such that $\dim_{\mathbb{K}}\hom (E_1 , E_2) < \infty$ is algebraic. In particular, $\T \cong \der(\operatorname{mod-}A)$, where $A=\operatorname{End}(\bigoplus_{i=1}^2 E_i)$.
\par\vspace{1ex}\noindent 

For a strong exceptional sequence with length greater than 2 we deal with realized triangulated categories, i.e. triangulated categories $\T$ admitting an exact functor $\mathsf{real}:\der(\mathcal{A}) \to \mathcal{S}$ for every heart $\mathcal{A}$ of a bounded t-structure on a full subcategory $\mathcal{S}$ of $\T$. In particular, it has been proven that all algebraic triangulated categories are realized (see Example \ref{es:real} for a discussion on examples of realized triangulated categories). By induction on the length of the exceptional sequence, we can prove the main result.

\par\vspace{1ex}\noindent 
\textbf{\ref{thm:genbon}. Theorem.} 
Let $\mathbb{K}$ be a field and let $\T$ be a realized $\mathbb{K}$-linear triangulated category with a full strong exceptional sequence $\langle E_1 , \dots , E_n \rangle$ such that $\bigoplus_i \hom(X,Y[i])$ is a finite-dimensional vector space for any $X,Y \in \T$. Then $\T \cong \der(\operatorname{mod-}A)$, where $A= \operatorname{End}(\bigoplus_{i=1}^n E_i)$.
\par\vspace{1ex}
\noindent \hrulefill
\par \vspace{1ex}
\noindent In Section 1, we recall some basic results on t-structures. Section 2 is devoted to the notion of compatible t-structures with respect to a semiorthogonal decomposition. Section 3 covers the needed knowledge on quivers, while Section 4 deals with filtered triangulated categories, introduced by Beilinson in \cite[Appendix A]{beilinson}. In Section 5, we introduce the concept of realized triangulated category and state the main theorem.
Appendix A generalizes a result on Yoneda extensions of exact categories.

\section{Some basic results on bounded t-structures}
In this section, we define t-structures and hearts, and state some classical results.

\begin{defn}
A \emph{t-structure} on a triangulated category $\T$ is a full subcategory $\T^{\le 0}$ closed by left shifts, i.e. $\T^{\le 0}[1] \subset \T^{\le 0}$, and such that for any object $E \in \T$ there is a distinguished triangle $A \to E \to B \to A[1]$, where $A \in \T^{\le 0}$ and $B \in \T^{\ge 1}:= (\T^{\le 0})^\bot$.

We remember that for any full subcategory $\mathcal{C} \subset \T$, we write $\mathcal{C}^\bot$ to mean the full subcategory whose objects are $Y$ such that $\hom(X,Y)=0$ for any $X \in \mathcal{C}$.

We will write $\T^{\le i}:=\T^{\le 0}[-i]$ and $\T^{\ge j} := \T^{\ge 1}[-j+1]$ for any $i,j$ integers. A t-structure is said to be \emph{bounded} if 
\[
\T = \bigcup_{i,j \in \mathbb{Z}} (\T^{\le i} \cap \T^{\ge j}).
\]
Moreover, the t-structure is \emph{non-degenerate} if $\bigcap_i \T^{\le i}=\bigcap_j \T^{\ge j}=0$.

The \emph{heart (of the bounded t-structure $\T^{\le 0}$)} is the additive category $\mathcal{A}:= \T^{\ge 0} \cap \T^{\le 0}$, and it is proven to be abelian. We define the \emph{homological dimension of $\mathcal{A}$ in $\mathcal{T}$}, denoted by $\dim_{\T}\mathcal{A}$, as the greatest integer $n$ such that $\hom(A,B[n])\ne 0$ for some $A,B \in \mathcal{A}$.
\end{defn}
\begin{nota}
Given a map $f:A\to B$ in a triangulated category, its cone will be denoted by $\cone(f)$. When there is no need to make $f$ explicit, we will write $\cone (A\to B)$.
\end{nota}

\begin{defnp} {\cite[Lemma 3.2]{bri}.}\label{def:heart}
Let $\T$ be a triangulated category. A heart (of bounded t-structure) on $\T$ is an additive category $\mathcal{A}$ satisfying the following properties: 
\begin{enumerate}
\item For any two objects $A,B \in \mathcal{A}$, $\hom(A,B[n])=0$ for every $n<0$.
\item Given an object $E \in \mathcal{T}$, we can find integers $k_1 > \dots >k_m$ and a filtration
\[
0=E_0 \to E_1 \to \dots \to E_{m-1}\to E_m=E
\]
such that $\cone (E_{i-1} \to E_{i})=A_i[k_i]$ for some $A_i \in \mathcal{A}$.
The \emph{cohomology objects} (with respect to $\mathcal{A}$) are defined as $H^{-k_i}(E):=A_i$. 
\end{enumerate} 

\end{defnp}
\begin{lem}\label{lem:bnondeg}
Every bounded t-structure $\mathcal{T}^{\le 0}$ is  non-degenerate.
In particular, the collection of functors $H^i$ is conservative and $H^i(E)=0$ for all $i>0$ (respectively $i<0$) if and only if $E \in \T^{\le 0}$ (respectively $\T^{\ge 0}$); this is \cite[Proposition 1.3.7]{bbd}.
\end{lem}
\begin{proof}
Let $E$ be in the intersection of all $\mathcal{T}^{\le i}$. Since $\T^{\le 0}$ is bounded, $E$ must be in $\T^{\le j} \cap \T^{\ge h}$ for some $j,h$. Then $E$ is in $\T^{\ge h}$, but also in $\T^{\le h-1}$. By definition, 

\[
\T^{\ge h} = \T^{\ge 1}[-h+1] = (\T^{\le 0})^{\bot} [-h+1]= (\T^{\le 0}[-h+1])^{\bot} = (\T^{\le h-1})^{\bot}.
\]
So $\hom(E,E)=0$, therefore $E$ is a zero object. In the same way one proves that also $\bigcap_i \mathcal{T}^{\ge i} =0$.
\end{proof}

\begin{lem} \label{lem:sesdisth}
Let $\mathcal{A} \subset \T$ be a heart. Then $0\to A \overset{\alpha}{\to} B \overset{\beta}{\to} C \to 0$ is a short exact sequence in $\mathcal{A}$ if and only if there exists a map $\gamma:C \to A[1]$ such that 
\[
\begin{tikzcd}
A \ar[r,"\alpha"]& B \ar[r,"\beta"]& C \ar[r,"\gamma"]& A[1]
\end{tikzcd}
\] is a distinguished triangle (cf. proof of \cite[Theorem 1.3.6]{bbd}).
\end{lem}
\begin{proof}
Notice that $H^i(A)=0$ for $i\ne 0$ and $H^0(A) =A $ for any $A \in \mathcal{A}$. We recall that, given a distinguished triangle $E \to F \to G \to E[1]$, we have an induced exact sequence
\[
\dots \to H^i(E) \to H^i (F) \to H^i(G) \to H^{i+1}(E) \to \dots
\]
in $\mathcal{A}$.

Let $0\to A \overset{\alpha}{\to} B \overset{\beta}{\to} C \to 0$ be a short exact sequence in $\mathcal{A}$ and set $C':=\cone (A\to B)$. The cohomology functors give rise to an exact sequence
\[
0 \to H^{-1}(C') \to A \overset{\alpha}{\to} B \to H^0(C') \to 0
\]
in $\mathcal{A}$. Since $\alpha$ is a monomorphism in $\mathcal{A}$, $H^{-1}(C')$ must be zero; on the other side, $H^0(C')\cong \coker (A\to B)= C$. The filtration in the definition proves that $C' \cong H^0(C')\cong C$, so we can choose $C'$ to be $C$ with the same map appearing in the distinguished triangle.

Conversely, let $A\overset{\alpha}{\to} B \overset{\beta}{\to} C \overset{\gamma}{\to} A[1]$ be a distinguished triangle with $A,B,C \in \mathcal{A}$. Then the cohomology functors show that there is an exact sequence
\[
0=H^{-1}(C) \to H^0(A) \overset{\alpha}{\to} H^0 (B) \overset{\beta}{\to} H^0(C) \to H^1(A)=0,
\]
concluding that $0\to A \overset{\alpha}{\to} B \overset{\beta}{\to} C \to 0$ is a short exact sequence.
\end{proof}
%

\section{Semiorthogonal decompositions and t-structures}

After recalling the notion of semiorthogonal decomposition, we define compatibility between t-structures with respect to such decomposition. In Theorem \ref{thm:decheart} we show how this situation gives rise to a global t-structure. As an application of the result, we study exceptional sequences and state Corollary \ref{cor:uniex}, which generalizes Bondal's theorem \cite[Theorem 6.2]{bondal} for exceptional sequences of length 2.

\begin{defn}
Let $\T$ be a triangulated category. A \emph{semiorthogonal decomposition} is a sequence of full triangulated subcategories $\mathcal{T}_1 , \mathcal{T}_2 , \dots, \mathcal{T}_n$ such that
\begin{enumerate}
\item $\hom(\mathcal{T}_i,\mathcal{T}_j)=0$ with $i>j$;
\item For any $E \in \T$, there is a filtration
\[
0 = E_{n} \to E_{n-1} \to \dots \to E_1 \to E_0= E
\]
such that $\cone(E_{i} \to E_{i-1}) \in \mathcal{T}_i$ for any $i \in \lbrace 1, \dots, n \rbrace$.
\end{enumerate}
In this situation, we will write $\mathcal{T}=\langle \mathcal{T}_1 , \mathcal{T}_2, \dots, \mathcal{T}_n\rangle$.
\end{defn}
\begin{rem}\label{rem:uniquecpt}
Item 1 entails that both the filtration and its cones are unique up to isomorphism and functorial, as observed in \cite[Remark 2.2]{mst}.
\end{rem}
\begin{defn}
Let $\mathcal{T}$ be a triangulated category. Given two full subcategories $\mathcal{X}$ and $\mathcal{Y}$ of $\mathcal{T}$, we define $\mathcal{X}* \mathcal{Y}$ to be the full subcategory of $\mathcal{T}$ whose objects are
\[
\lbrace Z \in \mathcal{T} \mid\text{there exists a distinguished triangle }X \to Z \to Y \to X[1]\text{, with }X \in \mathcal{X},Y \in \mathcal{Y}\rbrace.
\]
This construction gives rise to an operation $*$ between full subcategories of $\mathcal{T}$.
\end{defn}
\begin{prop}\emph{\cite[Lemma 1.3.10]{bbd}.}
The operation $*$ is associative.
\end{prop}
\begin{es}
Let $\mathcal{T}$ be a triangulated category.
Given a semiorthogonal decomposition $\mathcal{T}= \langle \mathcal{T}_1 , \dots , \mathcal{T}_n \rangle$, we can write
$\mathcal{T}= \mathcal{T}_n*\dots *\mathcal{T}_2*\mathcal{T}_1$. 
If we consider a t-structure $\mathcal{T}^{\le 0}$ on $\mathcal{T}$, we have $\mathcal{T}=  \mathcal{T}^{\le 0}*\mathcal{T}^{\ge 1} $.
\end{es}
%
%

\begin{defn}
Let $\T=\langle \T_1 , \T_2 \rangle$ be a semiorthogonal decomposition, $\T$ any triangulated category. Assume that $\T_i$ has a t-structure $\T_i^{\le 0}$ for $i=1,2$. Then $\T_1^{\le 0}$ and $\T_2^{\le 0}$ are \emph{compatible in $\T$} if $\hom(\T_1^{\le 0} , \T_2^{\ge 1}) = 0$. 

Denoted by $\mathcal{A}_1$ and $\mathcal{A}_2$ the hearts of $\mathcal{T}_1^{\le 0} $ and $\T^{\le 0}_2$ respectively, the \emph{relative dimension of $\mathcal{A}_1$ and $\mathcal{A}_2$ in $\T$} is the number
\[
\rdim_{\T}(\mathcal{A}_1 , \mathcal{A}_2) := \begin{cases}
\sup \left \lbrace m \in \mathbb{Z} \mid \hom(\mathcal{A}_1 , \mathcal{A}_2 [m])\ne 0 \right \rbrace &\text{if the set is nonempty}\\
-1 & \text{otherwise.}
\end{cases}
\]
Notice that, whenever the set above is nonempty, $\rdim_{\T}(\mathcal{A}_1, \mathcal{A}_2) \ge 0$ by compatibility. The reason why we have chosen the value $-1$ in case the set is empty will become clear reading the statement of Theorem \ref{thm:decheart}.
\end{defn}

\begin{thm}\label{thm:decheart} Let $\mathcal{T}$ be a triangulated category with a semiorthogonal decomposition $\T= \langle \T_1 , \T_2 \rangle$. Given two compatible t-structures $\T_1^{\le 0}$ and $\T_2^{\le 0}$ on $\T_1$ and $\T_2$ respectively, the full subcategory defined by  
\[
\T^{\le 0} := \mathcal{T}_2^{\le 0} * (\mathcal{T}_1^{\le 0}[1])
\]
is a t-structure.
Furthermore,
\begin{enumerate}
\item If $\T_1^{\le 0}$ and $\T_2^{\le 0}$ are bounded (respectively non-degenerate), then $\T^{\le 0}$ is bounded (respectively non-degenerate).
\item Let $\mathcal{A}_1$, $\mathcal{A}_2$ and $\mathcal{A}$ be the hearts associated to $\T_1^{\le 0}$, $\T_2^{\le 0}$ and $\T^{\le 0}$ respectively. Then 
\[
\mathcal{A}=\mathcal{A}_2 * (\mathcal{A}_1[1]).
\]
\item 
The following equality holds true whenever at least one of the two hearts $\mathcal{A}_1, \mathcal{A}_2$ is nonzero: \[
\dim_{\T} \mathcal{A} = \max \lbrace  \dim_{\T_1} \mathcal{A}_1 , \dim_{\T_2} \mathcal{A}_2 , \rdim_{\T} (\mathcal{A}_1 , \mathcal{A}_2)+1\rbrace.\]
\end{enumerate}
\end{thm}
\begin{proof}
Since $\T_i^{\le 0}[1] \subset \T_i^{\le 0}$ for $i=1,2$, it is clear that also $\T^{\le 0}$ is closed by left shifts. 
We aim to show that $\mathcal{T}= \mathcal{T}^{\le 0}* \mathcal{T}^{\ge 1}$, where $\mathcal{T}^{\ge 1}:= (\mathcal{T}^{\le 0})^{\bot}$. Notice that 
\[
\mathcal{T}= \mathcal{T}_2 * \mathcal{T}_1 = \mathcal{T}_2^{\le 0} *\mathcal{T}_2^{\ge 1}* (\mathcal{T}_1^{\le 0} [1]) * (\mathcal{T}_1^{\ge 1} [1]).\]
Since $\langle \mathcal{T}_1 , \mathcal{T}_2 \rangle$ is a semiorthogonal decomposition and compatibility holds, we have \[
\mathcal{T}_2^{\ge 1}*(\mathcal{T}_1^{\le 0} [1])=\lbrace X \oplus Y \mid X \in \mathcal{T}_2^{\ge 1} , Y \in \mathcal{T}_1^{\le 0} [1]\rbrace=(\mathcal{T}_1^{\le 0} [1])* \mathcal{T}_2^{\ge 1}.\] 
Therefore, $\mathcal{T}=\mathcal{T}_2^{\le 0} *(\mathcal{T}_1^{\le 0} [1])* \mathcal{T}_2^{\ge 1} * (\mathcal{T}_1^{\ge 1} [1])$. We claim that $\mathcal{T}_2^{\ge 1} * (\mathcal{T}_1^{\ge 1} [1])= \mathcal{T}^{\ge 1}$. 

Let $A \in \mathcal{T}_2^{\ge 1} * (\mathcal{T}_1^{\ge 1} [1])$. There exists a distinguished triangle
$A_2^{\ge 1}\to A \to A_1^{\ge 1}[1] \to A_2^{\ge 1}[1]$
with $A_2^{\ge 1} \in \T_2^{\ge 1}$ and $A_1^{\ge 1}[1] \in \T_1^{\ge 1}[1]$. Now let $B\in \T^{\le 0}$ and consider the distinguished triangle $B_2^{\le 0} \to B \to B_1^{\le 0}[1] \to B_2^{\le 0}[1]$, where $B_2^{\le 0} \in \T_2^{\le 0}$ and $B_1^{\le 0}[1] \in \T_1^{\le 0}[1]$. The two distinguished triangles introduced give rise to the following hom-exact sequences:
\[
\begin{tikzcd}[row sep=tiny,column sep=1.7em]
\dots \ar[r]& \hom(B,A_2^{\ge 1}) \ar[r]& \hom(B,A) \ar[r]& \hom (B,A_1^{\ge 1}[1]) \ar[r]& \dots\\
\dots \ar[r]& \hom(B_1^{\le 0}[1] , A_1^{\ge 1}[1]) \ar[r]& \hom(B,A_1^{\ge 1}[1]) \ar[r]& \hom(B_2^{\le 0} , A_1^{\ge 1}[1]) \ar[r]& \dots\\
\dots \ar[r]& \hom(B_1^{\le 0}[1] , A_2^{\ge 1}) \ar[r]& \hom(B,A_2^{\ge 1}) \ar[r]& \hom(B_2^{\le 0} , A_2^{\ge 1}) \ar[r]& \dots
\end{tikzcd}
\]
Since $\langle \T_1 ,\T_2\rangle$ is a semiorthogonal decomposition, $\hom(B_2^{\le 0} , A_1^{\ge 1}[1])=0$, and the properties of t-structures tell us that $\hom(B_1^{\le 0}[1] , A_1^{\ge 1}[1]) =0= \hom(B_2^{\le 0} , A_2^{\ge 1})$. By compatibility, we also have $\hom(B_1^{\le 0}[1] , A_2^{\ge 1})=0$.  Then the last two exact sequences prove that $\hom(B,A_1^{\ge 1}[1])=\hom(B,A_2^{\ge 1})=0$. The first exact sequence concludes that $\hom(B,A)=0$. Finally, $A \in \T^{\ge 1}$.

Conversely, if $A \in \T^{\ge 1}$, then there exists a distinguished triangle
\[
A^{\le 0} \to A \to A^{\ge 1} \to A^{\le 0}[1]
\] 
with $A^{\le 0} \in \T^{\le 0}$ and $A^{\ge 1} \in \mathcal{T}_2^{\ge 1} * (\mathcal{T}_1^{\ge 1} [1])$. Notice $A^{\le 0} \to A$ must be zero because $A \in \T^{\ge 1}$. Since $A^{\ge 1}$ cannot have a direct summand in $\T^{\le 0}$, we get that $A^{\le 0}=0$. In particular, $A=A^{\ge 1}$; as wanted, $\mathcal{T}_2^{\ge 1} * (\mathcal{T}_1^{\ge 1} [1])= \T^{\ge 1}$.
\begin{enumerate}
\item First, we deal with boundedness. Let $A \in \T$.
From the semiorthogonal decomposition, we get a distinguished triangle $A_2 \to A \to A_1[1]\to A_2[1]$ where $A_i \in \T_i$ for $i=1,2$.\footnote{The choice of $A_1[1]$ in the distinguished triangle is motivated by the fact that $\T^{\le 0} = \T_2^{\le 0} *(\T_1^{\le 0}[1])$.} Since $\T_i^{\le 0}$ is bounded, $A_i \in \T_i^{\le k_i} \cap \T_i^{\ge h_i}$ for some integers $k_i,h_i$. Let $k:= \max\lbrace k_1 , k_2\rbrace$ and $h:= \min \lbrace h_1, h_2\rbrace$.

By assumption, $A_i \in \T_i^{\le 0}[-k_i]$, so $A_i[k] \in \T_i^{\le 0}[k-k_i] \subseteq \T_i^{\le 0}$ being t-structures closed by left shifts. Therefore, $A[k] \in \T^{\le 0}$; in other words $A\in \T^{\le 0}[-k]=\T^{\le k}$.

Similarly, $A_i \in \T_i^{\ge 1}[1-h_i]$ implies $A_i[h-1] \in \T_i^{\ge 1}[1 -h_i+h-1] \subseteq \T_i^{\ge 1}$ (here we use the closure by right shifts). We conclude that $A[h-1] \in \T^{\ge 1}$, which means that $A\in \T^{\ge 1}[1-h] = \T^{\ge h}$. We have showed that $A \in \T^{\le k} \cap \T^{\ge h}$.
%
%
%
%
%
%

To prove non-degeneracy when $\mathcal{T}_1^{\le 0}$ and $\mathcal{T}_2^{\le 0}$ are non-degenerate, let $C \in \bigcap_j \mathcal{T}^{\le j}$. By Remark \ref{rem:uniquecpt}, we have $C= \cone (E \to F)$ for $E \in \bigcap_j \mathcal{T}_1^{\le j}$ and $F \in \bigcap_j \mathcal{T}_2^{\le j}$. By hypothesis, both intersections are zero, so $C \cong 0$ as wanted. The proof of $\bigcap_j \mathcal{T}^{\ge j}=0$ is analogous since $\T^{\ge 1} =\mathcal{T}_2^{\ge 1} * (\mathcal{T}_1^{\ge 1} [1])$.
\item For any $A \in \mathcal{A}$ we can find two distinguished triangles, according to the fact that $A \in \T^{\le 0}$ and $A[-1] \in \T^{\ge 1}=\mathcal{T}_2^{\ge 1} * (\mathcal{T}_1^{\ge 1} [1])$. Then Remark \ref{rem:uniquecpt} proves that $\mathcal{A}$ is exactly as described in the statement.

\item Let $A= \cone (A_1 \to A_2)$ and $B= \cone (B_1 \to B_2)$ be two objects of $\mathcal{A}$, with $A_i,B_i \in \mathcal{A}_i$, $i=1,2$. For any $m$, we consider the long exact sequence
\[
\dots \to \hom (A_1[1] , B[m]) \to \hom (A, B[m]) \to \hom (A_2,B[m]) \to \dots
\]
associated to the distinguished triangle $A_1 \to A_2 \to A \to A_1[1]$. By considering the first and the last term, we can create two exact sequences associated to $B_1 \to B_2\to B \to B_1[1]$:
\begin{gather*}
\dots \to \hom(A_1[1] , B_2[m]) \to \hom (A_1[1],B[m]) \to \hom(A_1[1], B_1[m+1]) \to \dots \\
\dots \to \hom(A_2, B_2[m]) \to \hom(A_2,B[m]) \to \hom(A_2,B_1[m+1]) \to \dots
\end{gather*}
Notice that $\hom(A_2,B_1[m+1])=0$ since $A_2 \in \T_2$ and $B_1 \in \T_1$. Taking
\[
\ell:=\max \lbrace \dim_{\T_1} \mathcal{A}_1 , \dim_{\T_2} \mathcal{A}_2  , \rdim_{\T} (\mathcal{A}_1 , \mathcal{A}_2)+1\rbrace,\] 
the exact sequences above prove that $\hom(A,B[m])=0$ for any $m>\ell$, so $\dim_{\T}\mathcal{A} \le \ell$. To conclude, it suffices to show that $\dim_{\T}\mathcal{A}\ge \ell$. 

We have two cases. If $\ell$ is realized by the homological dimension of $\mathcal{A}_1$ or $\mathcal{A}_2$, we notice that $\mathcal{A}_1[1], \mathcal{A}_2   \subset \mathcal{A}$ by item 2, so $\dim_{\T}\mathcal{A}\ge \ell$.

Assume $\ell=\rdim_{\T}(\mathcal{A}_1 , \mathcal{A}_2)+1$. If $0 < \ell <+\infty$, for some choices of $A_1[1]$ and $B_2$ in $\mathcal{A}$ we have $\hom(A_1[1],B_2[\ell])\ne 0$. Similarly, if $\ell=+ \infty$, there is a sequence $(a_n) \subset \mathbb{Z}$ such that $a_n \to +\infty$ and $\hom(A_1^{n}[1],B_2^{n}[a_n])\ne 0$ for any $a_n$ and some $A_1^{n}[1], B_2^{n} \in \mathcal{A}$. Since
item 2 tells us that $A_1[1], A_1^n[1] ,B_2 , B_2^n \in \mathcal{A}$, in both cases $\dim_{\T}\mathcal{A}$ cannot be less than $\ell$. If $\ell=0$, then $\ell$ is also equal to the homological dimensions of $\mathcal{A}_1$ or $\mathcal{A}_2$, and this possibility has already been addressed.  \qedhere
\end{enumerate}
\end{proof}

\begin{rem}
As already used in the last part of the proof, the constructed t-structure may not behave as wanted. For instance, using the notation of the statement, $\mathcal{A}_1$ is not contained in $\mathcal{A}$: we need to consider its shift $\mathcal{A}_1[1]$. 

One may think this shifting could be easily adjusted, but the requirement needed is incredibly strong. The first idea it comes to mind is to consider the t-structure $\mathcal{T}_1^{\le 1}= \T_1^{\le 0}[-1]$ instead of $\mathcal{T}_1^{\le 0}$. Indeed, if we ask $\T_1^{\le 1}$ and $\T_2^{\le 0}$ to be compatible, no shift will be involved, and in particular $\mathcal{A}_1, \mathcal{A}_2 \subset \mathcal{A}$. However, requiring $\mathcal{T}_1^{\le 1}$ and $\T_2^{\le 0}$ to be compatible implies that $\hom(\mathcal{A}_1, \mathcal{A}_2)=0$, which is generally too restrictive.
\end{rem}

\begin{rem}\label{rem:torsionpair}
Theorem \ref{thm:decheart} is incredibly linked to torsion pairs (for an introduction of the concept, we refer to \cite[Section I.2]{hrs}). Let $\T$ be a triangulated category with a semiorthogonal decomposition $\langle \T_1 , \T_2 \rangle$ and a t-structure $\T^{\le 0}$ such that $\T_i^{\le 0} = \T^{\le 0}\cap \T_i$ is a t-structure on $\T_i$ for $i=1,2$. If these t-structures are compatible in $\T$, Theorem \ref{thm:decheart} gives rise to a t-structure $\T_{\#}^{\le 0}$, which is different from $\T^{\le 0}$. Indeed, $E \in \T_1^{\le 0}\cap \T_1^{\ge 0}$ is  an object in $(\T_{\#}^{\le 0}\cap \T_{\#}^{\ge 0})[-1]$.

As a matter of fact, $\T_{\#}^{\le 0}$ gives rise to a heart which is a tilted version of the heart $\mathcal{A}$ of $\T^{\le 0}$. This is simply true by picking the couple $\mathcal{F}= \mathcal{A} \cap \T_1$ and $\mathcal{T}= \mathcal{A}\cap \T_2$, which is a torsion pair by \cite[Exercise 6.5]{msch}.
\end{rem}
\begin{rem}
Theorem \ref{thm:decheart} is very similar to \cite[Theorem 1.4.10]{bbd}, which constructs global t-structures via recollements instead of semiorthogonal decompositions. Let us explain this relation in detail.

First of all, we recall that any recollement gives rise to a semiorthogonal decomposition. We consider $\T=\langle \T_1 ,\T_2 \rangle$ a semiorthogonal decomposition and let $\T_i^{\le 0}$ a t-structure on $\T_i$ for $i=1,2$. Then, under their respective assumptions, from Theorem \ref{thm:decheart} we get the global t-structure $\T_2^{\le 0} * (\T_1^{\le 0}[1])$, while \cite[Theorem 1.4.10]{bbd} gives the t-structure $\T_2^{\le 0} * \T_1^{\le 0}$. In other words, the new result gives a tilted version of the old statement (see Remark \ref{rem:torsionpair}).

Moreover, the two theorems deal with different situations. Indeed, although it is possible to construct a left adjoint $i^*$ to the inclusion $i_*:\T_1 \to \T$ (i.e. $\T_1$ is \emph{left admissible}) and a right adjoint $j^*$ to the inclusion $j_!:\T_2 \to \T$ (i.e. $\T_2$ is \emph{right admissible}) by \cite[Lemma 3.1]{bondal},
in general a left (respectively right) admissible subcategory does not need to be right (respectively left) admissible. 
Conversely, a recollement does not ensure that the compatibility requirement is satisfied, since $\T_2^{\ge 1}$ is not necessarily equal to $\T^{\ge 1}\cap \T_2$. 

Concerning our studies, Theorem \ref{thm:decheart} is to be preferred because it computes the homological dimension of the obtained heart; this is crucial, especially for Corollary \ref{cor:uniex}.
\end{rem}

The definition of compatible t-structures so that Theorem \ref{thm:decheart} holds can be generalized to semiorthogonal decompositions of any length, but the requirement may result unnatural since we need to consider some shifting. 
\begin{defn}
Let $\T = \langle \T_1, \dots, \T_m\rangle$ and assume $\T_i$ has a t-structure $\T_i ^{\le 0}$ for $i=1, \dots , m$. Then all the t-structures are \emph{compatible} if $\hom(\T_i^{\le 0}[k-i-1],\T_k^{\ge 1})=0$ for any $k>i$.
\end{defn}
With this notion of compatibility, an analogous of Theorem \ref{thm:decheart} can be obtained by recursion. With the same notation of the definition above, if $\mathcal{A}_i$ is the heart of $\T_i^{\le 0}$, the heart $\mathcal{A}\subset \T$ built via Theorem \ref{thm:decheart} is described as
\[
\mathcal{A}=\mathcal{A}_m*\mathcal{A}_{m-1}[1]* \dots * \mathcal{A}_2[m-2]* \mathcal{A}_1[m-1].\]

\begin{es}[Exceptional sequence]\label{es:exc}Let $\mathbb{K}$ be a field and consider a $\mathbb{K}$-linear triangulated category $\T$. We recall that an \emph{exceptional object} is an object $E \in \T$ such that 
\[
\hom(E,E[n])=\begin{cases}
\mathbb{K} & \text{if }n=0\\
0 & \text{otherwise.}
\end{cases}
\] 
A sequence of exceptional objects $E_1, \dots , E_m \in \T$, such that $\hom(E_i , E_j[n])=0$ for any $i>j$ and all $n$, is called \emph{exceptional sequence}.  It is \emph{full} if $\T$ is generated by $E_1 , \dots , E_m$, i.e. if $\mathcal{T}$ is exact equivalent to the smallest full triangulated subcategory of $\T$ containing $E_1, \dots , E_m$.

Consider a $\mathbb{K}$-linear triangulated category with a  full exceptional sequence $E_1, \dots , E_m$ such that $\bigoplus_i \hom(A,B[i])$ is a finite-dimensional vector space for any $A,B \in \T$\footnote{In fact, it suffices to require this property for $A,B \in \lbrace E_1 , \dots , E_m \rbrace$.}. By \cite[\S 1.4]{huy}, it is known that such a full exceptional sequence gives rise to a semiorthogonal decomposition given by $\T_i = \lbrace \bigoplus_\ell E_i^{\oplus b_\ell} [\ell] : b_\ell \in \mathbb{N} \rbrace$. We will use the notation $\langle E_1, \dots , E_m \rangle$ to indicate exceptional sequences. Notice that on each $\mathcal{T}_i$ we can consider a bounded t-structure with heart $\mathcal{A}_i=  \lbrace E_i^{\oplus b}: b \in \mathbb{N} \rbrace$.

If the full exceptional sequence is also \emph{strong}\footnote{This condition can be weakened.}, i.e. $\hom(E_i , E_j[n])=0$ for any $i,j$ and $n\ne 0$, the above t-structures are compatible: indeed, taking $k>i$,
\[
\hom \left (\bigoplus_{\ell\ge 0} E_i^{\oplus b_\ell} [\ell] [k-i-1] , \bigoplus_{j< 0} E_k^{\oplus c_j}[j]\right ) =0.
\] 
Moreover, the t-structure induced on $\mathcal{T}$ is bounded. 
\end{es}

\begin{cor}\label{cor:uniex} Let $\mathbb{K}$ be a field. Any $\mathbb{K}$-linear triangulated category $\T$ with a full strong exceptional sequence $\langle E_1 , E_2 \rangle $ such that $\dim_{\mathbb{K}}\hom (E_1 , E_2) < \infty$ is algebraic. In particular, $\T \cong \der(\operatorname{mod-}A)$, where $A=\operatorname{End}(\bigoplus_{i=1}^2 E_i)$.
\end{cor}
\begin{proof}
Theorem \ref{thm:decheart} and Example \ref{es:exc} prove that $\T$ has a heart of dimension at most 1. By \cite[Theorem 3.2]{hubery}, $\T$ is algebraic. We conclude by \cite[Corollary 1.9]{orlov}.
\end{proof}
\begin{es}
By the previous corollary, $\der(\mathbb{P}_{\mathbb{K}}^1)$ is the unique $\mathbb{K}$-linear triangulated category with a full strong exceptional sequence $\langle E_1 , E_2 \rangle$ such that $\dim_{\mathbb{K}}\hom(E_1 , E_2)=2$.
\end{es}

\section{Quivers}\label{sec:quiver}
In order to study exceptional sequences of length greater than 2, we will need some basic knowledge on quivers. Here we will give a brief introduction, mostly following \cite[Section 5]{bondal}.

\begin{defn}
A \emph{quiver} is a quadruple $Q=(Q_0 , Q_1, s,t)$, where $Q_0$ is a set of vertices, $Q_1$ a set of arrows between vertices and $s,t:Q_1 \to Q_0$ are the maps indicating source and target respectively. A quiver is \emph{finite} if $Q_0$ and $Q_1$ are finite. It is \emph{ordered} if the vertices are ordered and for every arrow $a$, $s(a)\le t(a)$.

A \emph{path $p$ of length} $n$ is a sequence of arrows $a_1 , \dots  ,a_n \in Q_1$ such that $t(a_i)=s(a_{i+1})$. Moreover, with the same notation, $t(p):=t(a_n)$ and $s(p):=s(a_1)$. We also allow paths of length $0$: such paths are in correspondence with the vertices. Let $p,q$ be two paths. Then the composition of paths $q\circ p$ is defined to be the concatenated path whenever $s(q)=t(p)$.

Let $\mathbb{K}$ be a field. The \emph{path algebra} $\mathbb{K}Q$ is the $\mathbb{K}$-vector space with basis the paths. The product is described as follows:
\[
\lambda q \cdot \mu p = \begin{cases}
(\lambda \mu) \ q \circ p & \text{if }s(q)=t(p)\\
0 & \text{otherwise}
\end{cases}
\]
where $\lambda, \mu \in \mathbb{K}$ and $p,q$ are paths. In particular, paths of length $0$ are idempotents in $\mathbb{K}Q$.

If $S\subset \mathbb{K}Q$ is any subset, $(Q,S)$ is called \emph{quiver with relations} and its associated path algebra is given by $\mathbb{K}Q/\langle S \rangle$, where $\langle S \rangle$ is the ideal generated by $S$.
\end{defn}

Now, let us consider $A= \mathbb{K}Q /\langle S \rangle$ the path algebra associated to the quiver with relations $(Q,S)$. A left $A$-module is a vector space $V$ over $\mathbb{K}$ with the left action of the algebra $A$. This is also called \emph{representation of a quiver}. Using the paths of length 0, which are associated to the vertices of $Q$, then $V$, as a vector space, decomposes into a direct sum $\bigoplus_{i \in Q_0} V_i$, where $V_i$ is the vector space associated to the vertex $i$. Moreover, for every path $p \in A$, we get a linear operator $V_{s(p)} \to V_{t(p)}$.

When dealing with right $A$-modules, one can consider the \emph{opposite quiver} $Q^{\operatorname{op}}$ where $s,t$ are swapped with respect to $Q$. In other words, arrows go in the other direction, analogously to what happens with the notion of opposite category. As one expects, left modules associated to $(Q^{\operatorname{op}},S^{\operatorname{op}})$ are right modules of $A$.

In case the quiver $Q$ is finite and ordered, let $X_1, \dots , X_n$ be the vertices and $p_i$ the idempotent in $A$ associated to $X_i$. Every right $A$-module $V$ has a decomposition $V=\bigoplus_{i \in Q_0} G_i V = \bigoplus_{i \in Q_0} Vp_i$. 

Let us denote with $S_i$ the representation for which $G_j S_i = \delta_{ij} \mathbb{K}$, where $\delta_{ij}$ is the Kronecker delta, and all arrows are represented by the zero morphisms. Notice that for each right $A$-module $V$ we can create a filtration
\begin{equation}\label{eq:filt}
0=F^0 V\hookrightarrow F^1 V= G^1V \hookrightarrow F^2 V= \bigoplus_{j=1}^2 G_j V \hookrightarrow \dots \hookrightarrow F^{n-1} V= \bigoplus_{j=1}^{n-1} G_j V \hookrightarrow F^n V= V
\end{equation}
such that each quotient $F^i V / F^{i-1} V$ is a direct sum of copies of $S_i$. Projective modules are $P_i=p_i A$ and the decomposition $A= \bigoplus_{i=1}^n P_i$ holds. As a matter of fact,
\[
A= \hom_A (A,A) =\hom_A\left(\bigoplus_{i=1}^n P_i , \bigoplus_{i=1}^n P_i\right)= \bigoplus_{i,j} \hom(P_i , P_j). 
\]
These isomorphisms allow to interpret the arrows of a quiver as  morphisms between projective modules. In particular, being $A$ the path algebra of an ordered quiver, $\hom(P_i , P_j)=0$ for $i>j$. Furthermore, it is possible to consider the exact sequence
\begin{equation}\label{eq:sesproj}
0 \to F^{i-1} P_i \to P_i \to S_i \to 0
\end{equation}
for every $i =1, \dots , n$. Notice that $P_1 =S_1$.

Let $\T$ be a $\mathbb{K}$-linear algebraic triangulated category with a full strong exceptional sequence $\langle E_1 , \dots , E_n\rangle$. Then $A=\operatorname{End}(\bigoplus_{i=1}^n E_i)$ is the path algebra of an ordered and finite quiver with relations. In particular, the equivalence $F:\T\to \der(\operatorname{mod-}A)$ obtained in \cite[Corollary 1.9]{orlov} is such that $F(E_i)=P_i$, the projective modules of the path algebra $A$. 

\section{Filtered enhancements}
In this section, we explore the definition of filtered triangulated categories and give a fairly simple result that has not been found in the literature, namely if a triangulated category admits a filtered enhancement, then every full triangulated subcategory admits a filtered enhancement in a natural way (see Proposition \ref{prop:subfilt}). Main reference is \cite[Appendix A]{beilinson}. In Remark \ref{rem:fcat7}, we discuss the relation of filtered enhancements with realization functors.

\begin{defn} \label{def:filt}
Let us consider a quintuple $(\mathcal{F}, \mathcal{F}(\le 0) , \mathcal{F}(\ge 0) , s, \alpha)$, where $\mathcal{F}$ is a triangulated category, $\mathcal{F}(\le 0 )$ and $\mathcal{F}(\ge 0)$ are strict full triangulated subcategories, $s: \mathcal{F} \to \mathcal{F}$ is an exact equivalence and $\alpha: \id_{\mathcal{F}} \to s$ is a natural transformation. We set $\mathcal{F}(\le n)=s^n \mathcal{F}(\le 0)$ and $\mathcal{F}(\ge n)=s^n \mathcal{F}(\ge 0)$. In this picture, $\mathcal{F}$ is called \emph{filtered triangulated category} if it satisfies the following axioms:
\begin{description}
\item[fcat1] $\mathcal{F}(\le 0) \subset \mathcal{F}(\le 1)$ and $\mathcal{F}(\ge 1) \subset \mathcal{F}(\ge 0)$.
\item[fcat2] $\mathcal{F}= \bigcup_n \mathcal{F}(\le n)=\bigcup_n \mathcal{F}(\ge n)$.
\item[fcat3] $\hom(\mathcal{F}(\ge 1 ), \mathcal{F}(\le 0))=0$.
\item[fcat4] For any $X \in \mathcal{F}$ there exists a distinguished triangle $A \to X \to B \to A[1]$ where $A \in \mathcal{F}(\ge 1)$ and $B \in \mathcal{F}(\le 0)$; in other words, $\mathcal{F}=\mathcal{F}(\ge 1) * \mathcal{F}(\le 0)$.
\item[fcat5] For any $X \in \mathcal{F}$, it holds that $\alpha_{s(X)}=s(\alpha_X)$.
\item[fcat6] For any $X \in \mathcal{F}(\ge 1)$ and $Y \in \mathcal{F}(\le 0)$, $\alpha$ induces isomorphisms
\[
\hom (Y,X) \cong \hom(Y,s^{-1}X) \cong \hom (sY,X).
\]
\end{description}
A triangulated category $\T$ admits a \emph{filtered enhancement} if there exists a filtered triangulated category $\mathcal{F}$ such that $\T \cong \mathcal{F}(\le 0) \cap \mathcal{F}(\ge 0)$ in the sense of triangulated categories. With an abuse of notation, we will always assume that $\T=\mathcal{F}(\le 0) \cap \mathcal{F}(\ge 0)$.
\end{defn}

\begin{prop}\emph{\cite[Proposition A.3]{beilinson}.}\label{prop:filt}
Let $\mathcal{F}$ be a filtered triangulated category. Then the following assertions hold true:
\begin{enumerate}
\item The inclusion $i_{\le n}:\mathcal{F}(\le n) \to \mathcal{F}$ has a left adjoint $\sigma_{\le n}$, and the inclusion $i_{\ge n}:\mathcal{F}(\ge n) \to \mathcal{F}$ has a right adjoint $\sigma_{\ge n}$. In particular, these adjoints are exact (see, for instance, \cite[Proposition 1.41]{huy}).
\item There is a unique natural transformation $\delta: \sigma_{\le n} \to \sigma_{\ge n+1}[1]$ such that, for any $X \in \mathcal{F}$,
\[
\sigma_{\ge n+1} (X) \to X \to \sigma_{\le n}(X) \overset{\delta(X)}{\to} \sigma_{\ge n+1}(X)[1]
\]
is a distinguished triangle. Up to unique isomorphism, this is the only distinguished triangle $A \to X \to B \to A[1]$ with $A \in \mathcal{F}(\ge n+1)$ and $B \in \mathcal{F}(\le n)$.
\item For any two integers $m,n$, we have the following natural isomorphisms:
\[
\sigma_{\le m} \sigma_{\le n}\cong \sigma_{\le \min\lbrace m,n\rbrace} , \qquad \sigma_{\ge m}\sigma_{\ge n} \cong \sigma_{\le \max \lbrace m,n \rbrace}, \qquad \sigma_{\ge m} \sigma_{\le n} \cong \sigma_{\le n} \sigma_{\ge m}.
\]
\end{enumerate}
\end{prop}
\begin{proof}[Part of the proof]
We want to prove the first two isomorphisms of item 3, since it is the only part of the statement not considered in \cite{beilinson}. Being the reasoning analogous, let us focus just on the first isomorphism. Let $X \in \mathcal{F}$. If $m\ge n$, then $\mathcal{F}(\le m) \supset \mathcal{F}(\le n)$. We recall that $\sigma_{\le m} i_{\le m} \cong \id$ because the inclusion $i_{\le m}$ is fully faithful. Since $\sigma_{\le n}(X) \in \mathcal{F}(\le m)$, we simply have that $\sigma_{\le m}\sigma_{\le n} (X) \cong \sigma_{\le n} (X)$ by the natural isomorphism mentioned before. We conclude that $\sigma_{\le m} \sigma_{\le n}\cong\sigma_{\le n}$. 

Let $m\le n$, so that $\mathcal{F}(\le m)\subset \mathcal{F}(\le n)$. Then, by adjunction, we have the following isomorphisms for any $X \in \mathcal{F}$ and $Y \in \mathcal{F}(\le m)$:
\begin{equation*}
\begin{split}
\hom_{\mathcal{F}(\le m)} (\sigma_{\le m} \sigma_{\le n} (X) , Y)&\cong \hom_{\mathcal{F}} ( \sigma_{\le n} (X) , i_{\le m} (Y))\\
&\cong \hom_{\mathcal{F}}( X, i_{\le n} i_{\le m} (Y)) \\
&
\cong \hom_{\mathcal{F}} (X, i_{\le m} (Y)).
\end{split}
\end{equation*}
In particular, $\sigma_{\le m} \sigma_{\le n}$ is left adjoint to $i_{\le m}$. Since adjoints are determined up to a natural isomorphism, $\sigma_{\le m}\sigma_{\le n} \cong \sigma_{\le m}$ as wanted.
\end{proof}

\begin{rem}\label{rem:ssigma}
By item 2 of Proposition \ref{prop:filt}, we also have the following isomorphisms:
\[
s \sigma_{\le n} \cong \sigma_{\le n+1} s, \qquad s \sigma_{\ge n} \cong \sigma_{\ge n+1} s.
\]
\end{rem}
Let us set $\gr^n:= \sigma_{\le n}\sigma_{\ge n}$. This is not the definition used in \cite{beilinson}, but it will come in handy in the proof of the following statement.

\begin{prop}\label{prop:subfilt}
Let $\T$ be a triangulated category admitting a filtered enhancement $\mathcal{F}$. Then any full triangulated subcategory $\mathcal{S}$ of $\T$ has a filtered enhancement given by the full subcategory $\mathcal{G}$ of $\mathcal{F}$ with objects
\[
\lbrace X \in \mathcal{F} \mid s^{-n} \gr^n (X) \in \mathcal{S} \ \forall n \rbrace.\]
\end{prop}
\begin{proof}
First of all, we would like to show that $\mathcal{G}$ is a triangulated subcategory of $\mathcal{F}$. Notice that the shift functor of $\mathcal{F}$ obviously restricts to $\mathcal{G}$ since $s^{-n}\gr^n$ is exact, being composition of exact functors. Let us consider $X \to Y$ with $X,Y \in \mathcal{G}$. This gives a distinguished triangle $X \to Y \to Z \to X[1]$ in $\mathcal{F}$. We get that \[
s^{-n} \gr^n (X) \to s^{-n}\gr^n (Y) \to s^{-n}\gr^n(Z) \to s^{-n} \gr^n(X[1])
\]
is a distinguished triangle in $\T$, with $s^{-n} \gr^n (X)$ and $s^{-n} \gr^n (Y)$ objects of $\mathcal{S}$. This suffices to conclude that $s^{-n} \gr^n (Z) \in \mathcal{S}$, so that $Z \in \mathcal{G}$. Next, we set $\mathcal{G}(\le 0):= \mathcal{G}\cap \mathcal{F}(\le 0)$ and $\mathcal{G}(\ge 0):= \mathcal{G}\cap \mathcal{F}(\ge 0)$. We would like to prove that the autoequivalence $s: \mathcal{F} \to \mathcal{F}$ can be restricted to $\mathcal{G}$. Let $X \in \mathcal{G}$. Then, by Remark \ref{rem:ssigma}, we have 
\begin{equation*}
\begin{split}
s^{-n} \gr^n (sX) & = s^{-n} \sigma_{\le n}\sigma_{\ge n} s(X)\\
& \cong s^{-n} \sigma_{\le n} s \sigma_{\ge n-1} (X) \\
&\cong s^{-n} s \sigma_{\le n-1} \sigma_{\ge n-1}(X) \\
&= s^{-n+1} \gr^{n-1}(X) \in \mathcal{S}.
\end{split}
\end{equation*}
So we can restrict $s$ and create an exact autoequivalence $s:\mathcal{G}\to \mathcal{G}$, called $s$ as well by an abuse of notation. Of course, the restriction of $\alpha: \id_{\mathcal{F}} \to s$ gives us the required natural transformation and fcat5 is ensured. We set $\mathcal{G}(\le n)$ and $\mathcal{G}(\ge n)$ via $s$ as described in Definition \ref{def:filt}. Being $s$ an equivalence, we have the following
\[
\mathcal{G}(\ge n) = s^n (\mathcal{G}(\ge 0)) = s^n (\mathcal{G}\cap \mathcal{F}(\ge 0)) = s^n (\mathcal{G}) \cap s^n (\mathcal{F}(\ge 0)) = \mathcal{G} \cap \mathcal{F}(\ge n),
\]
and analogously $\mathcal{G}(\le n) = \mathcal{G}\cap \mathcal{F}(\le n)$. This immediately shows that fcat1,2,3,6 hold. As fcat5 has already been dealt with, it remains to show fcat4. In order to do that, we recall the distinguished triangle in item 2 of Proposition \ref{prop:filt}. Therefore, the statement is reduced to establish that the images of $\sigma_{\le n}$ and $\sigma_{\ge n}$ are in $\mathcal{G}(\le n)$ and $\mathcal{G}(\ge n)$ respectively, so that these functors are adjoints to the inclusions as in $\mathcal{F}$.
Let $X \in \mathcal{G}$ and consider $\sigma_{\le m}$. By item 3 of Proposition \ref{prop:filt} and Remark \ref{rem:ssigma} the following isomorphisms hold:
\begin{equation*}
\begin{split}
s^{-n} \gr^n (\sigma_{\le m} X) &= s^{-n} \sigma_{\le n} \sigma_{\ge n} \sigma_{\le m} (X) \\
&\cong s^{-n} \sigma_{\le n} \sigma_{\le m} \sigma_{\ge n} (X)\\ 
&\cong s^{-n} \sigma_{\le m} \sigma_{\le n} \sigma_{\ge n} (X)\\
&
\cong \sigma_{\le m-n} s^{-n} \sigma_{\le n} \sigma_{\ge n} (X) .
\end{split}
\end{equation*}
In particular, $s^{-n}\gr^n(\sigma_{\le m}X)\cong \sigma_{\le m-n} (A)$, where $A \in \mathcal{S}$. If $m-n\ge 0$, we have the following inclusions: \[
A \in \mathcal{S}\subset \T \subset \mathcal{F}(\le 0)\subset \mathcal{F}(\le m-n),\]
so $\sigma_{\le m-n}(A)=A$. If $m-n<0$, being $A \in \mathcal{F}(\ge 0)$ it holds that $\hom(A,\sigma_{\le m-n}(A))=0$ by fcat3. In particular, item 2 of Proposition \ref{prop:filt} entails that $\sigma_{\le m-n}(A)=0$. As wanted, $s^{-n}\gr^n(\sigma_{\le m}X)\in \mathcal{S}$, so that $\sigma_{\le m}X \in \mathcal{G}$. With a similar reasoning, one can prove that $\sigma_{\ge m} X \in \mathcal{G}$.
\end{proof}
The reason why filtered enhancements become of great interest is their relation with realization functors. 
\begin{defn}
Let $\T$ be a triangulated category. Given a heart (of a bounded t-structure) $\mathcal{A}\subset \T$, we call \emph{realization functor} (of $\mathcal{A}$ in $\T$) an exact functor $\mathsf{real}:\der(\mathcal{A}) \to \T$ such that $\mathsf{real}_{\mid \mathcal{A}}=\id_{\mathcal{A}}$.
\end{defn}
\begin{rem}\label{rem:fcat7}
In \cite[Appendix]{beilinson}, it is proven that every triangulated category with a filtered enhancement admits a realization functor for any heart. However, some authors point out that an additional requirement, called fcat7, may be necessary to provide the result (see \cite[Appendix A]{psa} for further details).  

For the sake of completeness, let us state this new axiom using the same notation of Definition \ref{def:filt}.
\begin{description}
\item[fcat7] Given any morphism $f: X \to Y$ in $\mathcal{F}$, the diagram
\[
\begin{tikzcd}[column sep=large]
\sigma_{\ge 1 }(X) \ar[r]\ar[d,"\alpha_{\sigma_{\ge 1}(Y)} \sigma_{\ge 1}(f)"] & X \ar[r]\ar[d,"\alpha_Y f"] &\sigma_{\le 0}(X)\ar[r,"\delta(X)"]\ar[d,"\alpha_{\sigma_{\le 0}(Y)} \sigma_{\le 0}f"] &\sigma_{\ge 1 }(X)[1] \arrow{d}{\alpha_{\sigma_{\ge 1}(Y)} \sigma_{\ge 1}(f)[1]}\\
s(\sigma_{\ge 1}Y)\ar[r]&s(Y)\ar[r]&s(\sigma_{\le 0}Y)\ar[r] &s(\sigma_{\ge 1}Y)[1]
\end{tikzcd}
\]
can be extended to a $3 \times 3$-diagram whose rows and columns are distinguished triangles. 
\end{description}

Once ensured that $\mathcal{F}$ satisfies fcat7, it is easy to prove that also $\mathcal{G}$ as defined in Proposition \ref{prop:subfilt} fulfills fcat7. This will be key in what follows.
\end{rem}

\section{Realized triangulated categories}

This section revolves around the unconventional notion of realized triangulated categories. After the definition, we will give some large classes of examples studied in the literature and prove a crucial result, Proposition \ref{prop:realf}. As an application, the generalization of Bondal's theorem \cite[Theorem 6.2]{bondal} is ensured for realized triangulated categories.
\begin{defn}

A triangulated category $\T$ is called \emph{realized} if for every heart $\mathcal{A}$ of every full triangulated subcategory $\mathcal{S} \subset \mathcal{T}$ there exists a {realization functor} $\mathsf{real}:\der(\mathcal{A}) \to \mathcal{S}\subset \T$.
\end{defn}
\begin{es}\label{es:real}\item 
\begin{enumerate}
\item Triangulated categories with a filtered enhancement are realized, as discussed in Proposition \ref{prop:subfilt} and Remark \ref{rem:fcat7}.
\item Algebraic triangulated categories are realized by \cite[Theorem 3.2]{kelvos}, where the first item is proved in detail in \cite[Section 4]{kelch}.
In fact, every algebraic triangulated category has a filtered enhancement, as proved in \cite[Proposition 3.8]{chenri}, but fcat7 has not been investigated.

\item Every triangulated category which is the underlying category of a stable derivator admits a filtered enhancement; this is the content of \cite{modoi}. 


Furthermore, topological triangulated categories obtained by stable combinatorial model categories are filtered by \cite[Example 4.2]{groth}. In particular, there are examples of triangulated categories with a filtered enhancement which are not algebraic.
\end{enumerate}
\end{es}

\begin{prop}\label{prop:realf}\emph{\cite[Corollary 2.8]{chenhanzhou}.}
Let $\T$ be a realized triangulated category and let $\mathcal{A}$ be a heart on $\T$. The following assertions are equivalent:
\begin{enumerate}
\item $\T$ has all the Ext groups of $\mathcal{A}$ according to Definition \ref{def:allext}.
\item The realization functor $\mathsf{real}:\der(\mathcal{A}) \to \T$ is fully faithful.
\item The realization functor is full.
\end{enumerate} 
Moreover, under such circumstances, $\mathsf{real}$ is an exact equivalence.
\end{prop}
\begin{proof} We start with $1. \Rightarrow 2$. Let $E,F \in \der(\mathcal{A})$ and consider $\tilde{E}=\mathsf{real}(E)$ and $\tilde{F}=\mathsf{real}(F)$. Then, by Proposition \ref{def:heart}, there exist integers $k_1 > \dots > k_m$, $j_1 > \dots j_n$ and filtrations 
\begin{gather*}
0= E_0 \to E_1 \to \dots \to E_{m-1} \to E_m = E\\
0=F_0 \to F_1 \to \dots \to F_{n-1}\to F_n=F
\end{gather*}
with $\cone(E_{i-1} \to E_i)=E^i\in \mathcal{A}[k_i]$ and $\cone(F_{h-1} \to F_h)=F^h \in \mathcal{A}[j_h]$. Since $\mathsf{real}$ is an exact functor, their images $\tilde{E}_i, \tilde{E}^i, \tilde{F}_h$ and $\tilde{F}^h$ give the same filtrations. We consider the exact hom-sequences
\[
\begin{tikzcd}[row sep=tiny,column sep=1.7em]
\dots \ar[r]& \hom(E^i,F_h) \ar[r]& \hom(E_i,F_h) \ar[r]& \hom (E_{i-1},F_h) \ar[r]& \dots\\
\dots \ar[r]& \hom(E_{i} , F_{h-1}) \ar[r]& \hom(E_{i},F_h) \ar[r]& \hom(E_{i},F^h) \ar[r]& \dots
\end{tikzcd}
\]
From these sequences, an induction on $i$ and $h$ proves that $\hom(E_{i},F_h)\cong \hom(\tilde{E}_{i} , \tilde{F}_h)$, so finally $\hom(E,F)\cong \hom(\tilde{E}, \tilde{F})$ as wanted. Notice the base case is ensured since $\T$ has all the Ext groups of $\mathcal{A}$.

The implication $2. \Rightarrow 3.$ is trivial. We deal with $3. \Rightarrow 1$. In order to do that, we use item 3 and 4 of Proposition \ref{prop:yonext}, remembering Proposition \ref{prop:der}. Let $A,B \in \mathcal{A}$. Since $f_{1,A,B}: \ext^1(A,B) \to \hom_{\T} (A,B[1])$ is an isomorphism by Dyer's Theorem \ref{thm:dyer}, we know that $f_{2,A,B}$ is injective. Moreover, since $\mathsf{real}$ is full, $f_{2,A,B}$ is surjective, thus an isomorphism. The same reasoning proves that $f_{n,A,B}$ is an isomorphism for every $n$, concluding that $\T$ has all the Ext groups of $\mathcal{A}$.

We now assume that $\mathsf{real}$ is fully faithful and prove that it is also an equivalence. Let $E \in \T$. By Proposition \ref{def:heart}, there are a sequence of integers $k_1 > \dots > k_m$ and a filtration
\[
0 = E_0 \to E_1 \to \dots \to E_{m-1} \to E_m = E
\]
such that $\cone (E_{i-1} \to E_i)=E^i \in \mathcal{A}[k_i]$. We prove by induction on $m$ that $E$ is in the essential image of $\mathsf{real}$. If $m=0$, there is nothing to prove. If $m>0$, then by induction hypothesis $E_{m-1}= \mathsf{real}(\hat{E}_{m-1})$. Obviously, $E^m= \mathsf{real}(\hat{E}^m)$ because all shifts of the heart $\mathcal{A}$ are in the essential image of $\mathsf{real}$. By the filtration, $E=\cone(E^m[-1] \to E_{m-1})$. The map associated to this cone is the image of a unique map $f:\hat{E}^m[-1] \to \hat{E}_{m-1}$ in $\der(\mathcal{A})$ because $\mathsf{real}$ is fully faithful.  We consider its cone $\cone(f)$. Since $\mathsf{real}$ is exact, we obtain an isomorphism $\mathsf{real}(\cone(f)) \cong E$. 
\end{proof}
\begin{rem}
As a corollary, it is immediately proven that not all hearts in a derived category have all the Ext groups. Indeed, in $\der(\mathbb{P}^1)$ one can show that $\mathcal{A}=\lbrace \mathcal{O}_{\mathbb{P}^1}^{\oplus a_0}[2] \oplus \mathcal{O}_{\mathbb{P}^1}(1)^{\oplus a_1} \mid a_0 , a_1\ge 0 \rbrace$ gives a heart (this can be done by applying Theorem \ref{thm:decheart}). As highlighted in \cite[Exercise 5.3]{msch}, $\der(\mathcal{A})\cong \der (\operatorname{pt})^{\oplus 2}$ is not equivalent to $\der(\mathbb{P}^1)$, so $\der(\mathbb{P}^1)$ cannot have all the Ext groups of $\mathcal{A}$. 

With a different approach, notice that $\mathcal{A}\owns \mathcal{O}_{\mathbb{P}^1}[2] \to \mathcal{O}_{\mathbb{P}^1}(1)[2]\in \mathcal{A}[2]$ does not factor through an object in $\mathcal{A}[1]$, and therefore Corollary \ref{cor:extgroups} proves that $\der(\mathbb{P}^1)$ does not have all the Ext groups of $\mathcal{A}$. 
\end{rem}

\begin{rem}
Let $\mathbb{K}$ be a field and consider a realized $\mathbb{K}$-linear triangulated category $\T$ with a full strong exceptional sequence $\langle E_1 , \dots , E_n\rangle$. Then we can consider the heart $\mathcal{A}$ on $\T$ obtained according to Theorem \ref{thm:decheart} and Example \ref{es:exc}, giving rise to a realization functor $\der(\mathcal{A}) \to \T$. One would like to prove that such functor is in fact an equivalence, so that \cite[Corollary 1.9]{orlov} can be applied to ensure the generalization of Bondal's result \cite[Theorem 6.2]{bondal}. However, when $n>2$, it is not said that $\T$ has all the Ext groups of $\mathcal{A}$; for instance, if $n=3$,
 \[
\mathcal{A} \owns E_1[2] \overset{f}{\longrightarrow} E_3[2] \in \mathcal{A}[2]\]
does not necessarily factor through $\mathcal{A}[1]$. In general, we would have $f \notin \ext^2_{\mathcal{A}}(E_1[2],E_3)$ by Proposition \ref{prop:yonext}, item 1. For example, consider the quiver obtained by the following vertices and arrows:
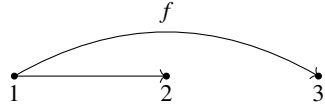
\begin{center}
\begin{tikzpicture}[font=\small]
\draw plot [mark=*, mark size=1.1] coordinates{(0,0)};
\node[below] at (0,0) {$1$};
\draw plot [mark=*, mark size=1.1] coordinates{(2,0)};
\node[below] at (2,0) {$2$};
\draw plot [mark=*, mark size=1.1] coordinates{(4,0)};
\node[below] at (4,0) {$3$};
\draw[->] (0,0)--(2,0);
\draw [->,bend left](0,0) to node[pos=0.5,above]{$f$} (4,0);
\end{tikzpicture}
\end{center}

In order to resolve this issue, we recall what already discussed in Remark \ref{rem:torsionpair}. If the length of the exceptional sequence is 2, the heart obtained by Theorem \ref{thm:decheart} is a tilt of $\operatorname{mod-}A$, where $A=\operatorname{End}(\bigoplus_{i=1}^2 E_i)$. As we will see, the same idea can be used to prove the general case.
\end{rem}

\begin{thm}\label{thm:genbon}
Let $\mathbb{K}$ be a field and let $\T$ be a realized $\mathbb{K}$-linear triangulated category with a full strong exceptional sequence $\langle E_1 , \dots , E_n \rangle$ such that $\bigoplus_i \hom(X,Y[i])$ is a finite-dimensional vector space for any $X,Y \in \T$. Then $\T \cong \der(\operatorname{mod-}A)$, where $A= \operatorname{End}(\bigoplus_{i=1}^n E_i)$.
\end{thm}
\begin{proof} 
We will prove the statement by induction on $n$, the length of the exceptional sequence. The base case $n=2$ is already taken care of by Corollary \ref{cor:uniex}.

If $n>2$, we write $\T= \langle \tilde{\T} , E_n \rangle$. By induction hypothesis, there exists an exact equivalence $\varphi: \der(\operatorname{mod-}\tilde{A})\to \tilde{\T}$ with $\tilde{A}=\operatorname{End}(\bigoplus_{i=1}^{n-1} E_i)$. We divide the proof in two parts:
\begin{enumerate}
\item The t-structures associated to $\varphi(\operatorname{mod-}\tilde{A})$ and $E_n$ are compatible. By Theorem \ref{thm:decheart}, we obtain a heart $\mathcal{A}$ on $\T$.
\item $\T$ has all the Ext groups of $\mathcal{A}$.
\end{enumerate}
Once both items are ensured, Proposition \ref{prop:realf} can be applied, proving that $\T\cong \der(\mathcal{A})$, and an application of \cite[Corollary 1.9]{orlov} will complete the proof.

From  \eqref{eq:filt}, every object $X \in \operatorname{mod-}\tilde{A}$ has an associated filtration
\[
0=F^0 X\hookrightarrow F^1 X \hookrightarrow \dots \hookrightarrow F^{n-2} X \hookrightarrow F^{n-1} X=X
\]
where $F^{k}X/F^{k-1}X$ is a direct sum of copies of $S_k$. Moreover, for each $P_k$ there is a short exact sequence $0 \to F^{k-1} P_k \to P_k \to S_k \to 0$ by \eqref{eq:sesproj}. In particular, $S_1 =P_1$.

Let us deal with 1. In order to prove it, it suffices to show that $\hom(\varphi(X),E_n[m])=0$ for every $m\le -1$ and $X \in \operatorname{mod-}\tilde{A}$. This can be done by induction on $k$, requiring $F^k X=X$. If $k=1$, $F^1 X $ is in fact a direct sum of copies of $P_1=\varphi^{-1}(E_1)$, so the claim holds.

If $k>1$, notice that the short exact sequence $0\to F^{k-1} P_k \to P_k \to S_k\to 0$ is associated to a distinguished triangle in $\T$, so it gives rise to the hom-sequence
\[
\hom (\varphi(F^{k-1} P_k)[1] , E_n[m])\to \hom (\varphi(S_k) , E_n[m]) \to \hom(E_k , E_n [m]).
\]
By induction, $\hom(\varphi (F^{k-1} P_k) [1] , E_n [m])=0$, while $\hom(E_k , E_n[m])=0$ by hypothesis. Therefore, $\hom(\varphi(S_k) , E_n[m])=0$. We now consider $X=F^k X$ and the distinguished triangle \[F^{k-1} X \to X\to X/F^{k-1} X \to F^{k-1} X[1]\] obtained by the filtration. From the associated hom-sequence, $\hom(\varphi(X), E_n[m])=0$ since the same holds for $F^{k-1} X$ and $X/F^{k-1} X$, the last one being a direct sum of copies of $S_k$.

It remains to prove item 2. According to Corollary \ref{cor:allext}, we will prove by induction on $m$ that $\hom(\varphi(X),E_n[m])\cong \ext_{\mathcal{A}}^m(\varphi(X),E_n)$ with $\varphi(X) \in \varphi(\operatorname{mod-}\tilde{A})[1] \subset \mathcal{A}$. The cases $m=0,1$ are true since $\mathcal{A}$ is a heart. Let $m>1$. By Proposition \ref{prop:yonext}, it holds that $\ext^m(E_k[1] , E_n)\subset \hom(E_k[1],E_n[m])=0$, and therefore $\ext^m(E_k[1] , E_n)=0$. 
Let us consider the distinguished triangle $F^{k-1} P_k \to P_k \to S_k \to F^{k-1} P_k [1]$.
Applying $\hom (\varphi( -) , E_n[m])$, we get the following commutative diagram
\begin{equation}\label{eq:indsimple}
\begin{tikzcd}[column sep=small]
\ext^{m-1}(E_k[1] , E_n) \ar[r]\ar[d,"\cong"]&\ext^{m-1} (\varphi(F^{k-1} P_k)[1],E_n) \ar[r]\ar[d,"\cong"] & \ext^m (\varphi(S_k)[1] , E_n)\ar[r]\ar[d] & 0\\
\hom(E_k[2], E_n[m]) \ar[r]& \hom(\varphi(F^{k-1} P_k)[2], E_n[m]) \ar[r]& \hom(\varphi(S_k)[1], E_n[m])\ar[r]&0
\end{tikzcd}
\end{equation}
proving that $\ext^m (\varphi(S_k)[1] , E_n) \cong \hom(\varphi(S_k)[1] , E_n[m])$ for every $k$ (use, for instance, the five lemma). 

Now, we proceed by induction on the length of the filtration. If $X=F^1 X$, there is nothing to prove since $F^1 X$ is a sum of copies of $S_1 =E_1$, and therefore $\hom(\varphi(F^1 X)[1] , E_n[m])=0$ since $m>1$. If $X=F^k X$, we consider the short exact sequence
$0 \to F^{k-1} X \to X \to X/F^{k-1} X \to 0$. Then we get the following diagram:
\[
\begin{tikzcd}[row sep=small]
\ext^{m-1}(\varphi(F^{k-1} X)[1] , E_n ) \ar[d]\ar[r,"\cong"]& \hom(\varphi(F^{k-1} X)[2] , E_n [m]) \ar[d]\\
\ext^m(\varphi(X/F^{k-1}X)[1] , E_n)\ar[d]\ar[r,"\cong"]
& \hom(\varphi(X/F^{k-1}X)[1] , E_n[m])\ar[d]\\
\ext^m(\varphi(X)[1] , E_n) \ar[d]\ar[r,"f_k"]& \hom(\varphi(X)[1] , E_n[m]) \ar[d]\\
\ext^m (\varphi(F^{k-1} X)[1] , E_n) \ar[d]\ar[r,"\cong"] & \hom (\varphi(F^{k-1} X)[1] , E_n [m]) \ar[d]\\
\ext^{m+1}(\varphi(X/F^{k-1} X)[1] , E_n)\ar[r,"g_k"]&\hom(\varphi(X/F^{k-1} X) , E_n[m]).\\
\end{tikzcd}
\]
To show that $f_k$ is an isomorphism, it suffices to apply the five lemma whenever $g_k$ is a monomorphism. More strongly, we claim that $g_k$ is an isomorphism. The idea is exactly the one seen above with the diagram \eqref{eq:indsimple}. In order to prove that \[\ext^{m+1} (E_k[1] , E_n) \subset \hom(E_k[1], E_n[m+1])=0,\] 
we will check that $\ext^{m} (E_k[1] , Y) \cong \hom(E_k[1], Y[m])$ for any $Y \in \mathcal{A}$, and conclude by item 3 of Proposition \ref{prop:yonext}. This is in fact true. Indeed, notice that \[
\ext^m (E_k[1] , \varphi(X)[1])=\hom (E_k[1] , \varphi(X)[m+1])=0\] for any $X \in \operatorname{mod-}\tilde{A}$ because $E_k$ is projective in $\varphi(\operatorname{mod-}\tilde{A})$. Furthermore, as remarked before \eqref{eq:indsimple}, $\ext^m(E_k[1] , E_n)=\hom(E_k[1],E_n[m])=0$. We conclude that \[
\ext^m(E_k[1] , Y)=\hom(E_k[1] , Y[m])=0\] 
since any $Y\in \mathcal{A}$ is the extension of a direct sum of copies of $E_n$ and an object $\varphi(X)[1] \in \varphi(\operatorname{mod-}\tilde{A})$.
\end{proof}

\appendix
\section{Yoneda extensions in a triangulated category}
A necessary remark to prove Hubery's main result is that, for a heart $\mathcal{A}$ in a triangulated category $\T$, if $\dim_{\T}(\mathcal{A})\le 1$, then $\dim_{\der(\mathcal{A})}(\mathcal{A})\le 1$ (see \cite[Section 3]{hubery}). This appendix aims to generalize this observation, providing results on Yoneda extensions in any triangulated category. 

First of all, we want to recall a theorem by Dyer, as it will give the desired generality for Proposition \ref{prop:yonext}. For this reason, let us give the definition of exact category according to Quillen \cite{quillen}.
\begin{defn}
An \emph{exact category} $\mathcal{A}$ is a full extension closed additive subcategory of an abelian category $\mathcal{B}$. A \emph{conflation} (or short exact sequence) is given by a short exact sequence in $\mathcal{B}$ contained in $\mathcal{A}$.
\end{defn}

\begin{thm}[Dyer]\emph{\cite{dyer}}.\label{thm:dyer} Let $\mathcal{A}$ be a full extension closed additive subcategory of a triangulated category $\T$ such that $\hom(A,B[-1])=0$ for any $A,B \in \mathcal{A}$. 

Then $\mathcal{A}$ has a natural exact structure, given by defining $0 \to A \to B \to C\to 0$ a conflation if $A \to B \to C \to A[1]$ is a distinguished triangle in $\T$ for some $C \to A[1]$. This association gives rise to a natural isomorphism $\ext^1_\mathcal{A}(A,B) \cong \hom_{\mathcal{T}}(A,B[1])$ for all $A,B \in \mathcal{A}$.
\end{thm}
\begin{rem}
A heart $\mathcal{A}$ in a triangulated category $\T$ satisfies the requirements of Theorem \ref{thm:dyer} thanks to Lemma \ref{lem:sesdisth}. 
\end{rem}
\begin{defn}[Yoneda extensions]\label{defn:extn}
Let $\mathcal{A}$ be an abelian category. The elements of the group $\ext^n(A,B)\cong \hom_{\der(\mathcal{A})}(A,B[n])$ are \emph{$n$-extensions} for $n>0$, i.e. exact sequences
\[
\mathbf{X}: 0 \to B \to X_1 \to \dots \to X_n \to A \to 0
\]
under the equivalence relation generated by identifying two exact sequences $\mathbf{X}$, $\mathbf{Y}$ if there is a family of morphisms $\psi= \lbrace \psi_1 , \dots , \psi_n \rbrace$ satisfying the following commutative diagram
\[
\begin{tikzcd}
0\ar[r] &B\ar[d,"\id"] \ar[r]&X_1\ar[d,"\psi_1"] \ar[r]& \dots \ar[r]& X_n \ar[d,"\psi_n"] \ar[r]&A\ar[r]\ar[d,"\id"] & 0\\
0\ar[r] &B \ar[r]&Y_1 \ar[r]& \dots \ar[r]& Y_n  \ar[r]&A\ar[r]& 0
\end{tikzcd}
\]
(cf. \cite[Theorem III.5.5]{gelfmani}). For $n=0$, $\ext^0(A,B)\cong \hom_{\der(\mathcal{A})}(A,B)\cong \hom_{\mathcal{A}}(A,B)$.

The \emph{Yoneda product} is given by maps $Y^{n,m}_{A,B,C}:\ext^n (A,B) \times \ext^m(B,C) \to \ext^{n+m}(A,C)$ for any $n,m\ge 0$ and any $A,B,C \in \mathcal{A}$. For $n,m \ge 1$, the Yoneda product is the map

\[
\begin{tikzcd}
\big(\mathbf{X}: 0 \to B \to X_1 \to \dots \to X_n \to A \to 0 \ , \ \mathbf{Y}: 0 \to C \to Y_1 \to \dots \to Y_m \to B \to 0 \big)\ar[d,mapsto]\\
\mathbf{Y}\cdot \mathbf{X}: 0 \to C \to Y_1 \to \dots \to Y_m \to X_1 \to \dots \to X_n \to A \to 0.
\end{tikzcd}
\]
If $n=m=0$, the product is simply the composition of maps. 
The case $n> 0$ and $m=0$ requires a more sophisticated definition.
Let $X_1 \in \ext^1(K,B)$ and $g:B \to C$. Then $g \cdot X_1$ is described by the following commutative diagram
\begin{equation}\label{eq:Xg}
\begin{tikzcd}
0 \ar[r] &B \ar[r]\ar[d,"g"]& X_1 \ar[d]\ar[r] &K \ar[r]\ar[d,"\id"] &0\\
0 \ar[r] & C \ar[r]& g \cdot X_1 \ar[r] &K \ar[r]&0
\end{tikzcd}
\end{equation}
where $g\cdot X_1$ is the pushout of $g$ and $B \to X_1$. Now, considering an $n$-extension 
\[
\mathbf{X}: 0 \to B \to X_1 \to X_2 \to \dots \to X_n \to A \to 0
\]
and $g:B \to C$, the Yoneda product is given by substituting $0 \to B \to X_1$ with $0 \to C \to g \cdot X_1$:
\[
g \cdot \mathbf{X} : 0 \to C \to g \cdot X_1 \to X_2 \to \dots \to X_n \to A \to 0.
\]
Dually, one can describe the case $n=0$ and $m>0$.
The Yoneda product so defined behaves according to the composition of maps (up to shift) \[\hom_{\der(\mathcal{A})} (A,B[n]) \times \hom_{\der(\mathcal{A})} (B,C[m]) \to \hom_{\der(\mathcal{A})} (A,C[n+m]).\] 

The structure of abelian group of $\hom_{\der(\mathcal{A})}(A,B[n])$ can be considered on $\ext^n(A,B)$ via the \emph{Baer sum}, described as follows. Let $\mathbf{X},\mathbf{Y}\in \ext^n(A,B)$. Consider the direct sum of the long exact sequences
\[
\mathbf{X} \oplus \mathbf{Y}:0 \to B\oplus B \to X_1 \oplus Y_1 \to \dots \to X_n \oplus Y_n \to A \oplus A \to 0,
\]
the diagonal map $\Delta_A=\left( \begin{smallmatrix} \id \\ \id \end{smallmatrix}\right) : A  \to A \oplus A$ and the codiagonal map $\nabla_B = (\id \ \id) : B\oplus B \to B$. Then the Baer sum is given by $\mathbf{X} + \mathbf{Y}:=\nabla_B \cdot (\mathbf{X}\oplus \mathbf{Y})\cdot \Delta_A$.

The \emph{(absolute) homological dimension of $\mathcal{A}$}, denoted by $\dim \mathcal{A}$, is the greatest integer $n$ such that $\ext^n(A,B) \ne 0$ for some $A,B \in \mathcal{A}$.
\end{defn}
\begin{rem}
Last definition can be generalized to any exact category $\mathcal{A}$, where an $n$-extension is a sequence
\[
\begin{tikzcd}
0 \ar[r]& B \ar[r,"\xi_0"]& X_1 \ar[r,"\xi_1"] & X_2 \ar[r,"\xi_2"]&\dots \ar[r,"\xi_{n-1}"]& X_n \ar[r,"\xi_n"]& A \ar[r]& 0
\end{tikzcd}
\]
such that, for $i=1, \dots , n-1$, $\xi_i$ factor through an object $C_i \in \mathcal{A}$ and
\[
0 \to B \to X_1 \to C_1 \to 0, \quad 0 \to C_1 \to X_2 \to C_2 \to 0, \quad \dots, \quad 0 \to C_{n-1} \to X_n \to A \to 0
\]
are conflations. In particular, $C_i = \im \xi_i=\ker \xi_{i+1}$.
\end{rem}

\begin{prop}\label{prop:yonext}\emph{\cite[Lemma 2.1]{chenhanzhou}.}
Let $\mathcal{A}$ be a heart of a triangulated category $\mathcal{T}$. More generally, let $\mathcal{A}$ be an exact subcategory of $\T$ as in Dyer's Theorem \ref{thm:dyer}. Then there is a well-defined map $f_{n,A,B}:\ext^n(A,B) \to \hom_{\T}(A,B[n])$ for any $A,B \in \mathcal{A}$ and $n\ge 0$. The following facts are true.
\begin{enumerate}
\item  The image of $f_{n,A,B}$ is given by all the maps $A \to B[n]$ factoring as 
\[A \to C_{n-1} [1] \to \dots \to C_1[n-1]\to B[n]\]
for some $C_i \in \mathcal{A}$, $i\in\lbrace 1, \dots , n-1 \rbrace$. 
\item The Yoneda product is sent to composition as expected: therefore, $f_{n,-,-}$ is a natural transformation and $f_{n,A,B}$ is a group homomorphism with respect to the Baer sum on $\ext^n(A,B)$.
\item If $f_{n-1,A,B}$ is an isomorphism for any $B\in \mathcal{A}$, then $f_{n,A,B}$ is injective. 
\item Let $g_{n,A,B}: \ext^n(A,B) \to \hom_{\T}(A,B[n])$ be a map for any $n\ge 0$ and $A,B \in \mathcal{A}$. If $g_{1,A,B}$ is the natural isomorphism of Theorem \ref{thm:dyer} and the Yoneda product is sent to composition, then $g_{n,A,B}=f_{n,A,B}$.
\end{enumerate}

\end{prop}
\begin{proof}
For $n=0$, $f_{0,A,B}: \hom_{\mathcal{A}} (A,B) \to \hom_{\mathcal{T}}(A,B)$ is an isomorphism since $\mathcal{A}$ is a full subcategory of $\T$. Let $n>0$ and consider $\mathbf{X}$ an exact sequence
\[
\begin{tikzcd}
0 \ar[r]& B \ar[r,"\xi_0"]& X_1 \ar[r,"\xi_1"] & X_2 \ar[r,"\xi_2"]&\dots \ar[r,"\xi_{n-1}"]& X_n \ar[r,"\xi_n"]& A \ar[r]& 0.
\end{tikzcd}
\]
To $\mathbf{X}$ we can associate short exact sequences 
\[
\begin{tikzcd}[row sep=0.3em]
0\ar[r]& B=\im \xi_0 \ar[r]& X_1 \ar[r]& \im \xi_1 \ar[r]& 0\\
0 \ar[r]& \im \xi_1 \ar[r]& X_2 \ar[r]& \im \xi_2 \ar[r]& 0\\
&&\vdots&&\\
0\ar[r]& \im \xi_{n-1} \ar[r]& X_n \ar[r]& \im \xi_n= A\ar[r]& 0 
\end{tikzcd}
\]
which are associated to distinguished triangles; therefore, we can consider a map 
\[
A \to \im \xi_{n-1}[1]\to \dots \to \im \xi_2[n-2] \to \im \xi_1[n-1] \to B[n].\]
We need to show that if $(\mathbf{X}, \xi)$ and $(\mathbf{Y},\eta)$ give the same $n$-extension, then the associated map $A\to B[n]$ obtained is the same. Without loss of generality, assume there is a family of morphisms $\psi$ as in Definition \ref{defn:extn}. Then for each $i\in \lbrace 0, \dots, n-1\rbrace$ we have
\begin{equation*}
\begin{tikzcd}
\im \xi_i \ar[r]\ar[d,"\varphi_i"] & X_{i+1} \ar[r]\ar[d,"\psi_{i+1}"]& \im \xi_{i+1} \ar[r]\ar[d,"\varphi_{i+1}"]&\im \xi_i[1]\arrow{d}{\varphi_i[1]}\\
\im \eta_i \ar[r] & Y_{i+1} \ar[r]& \im \eta_{i+1} \ar[r]&\im \eta_i[1]
\end{tikzcd}
\end{equation*}
where $\varphi_i$ is obtained by the universal property of the kernel. In order to prove that the middle square is commutative, we notice that 
\begin{equation*}
\begin{split}
X_{i+1} \to \im \xi_{i+1} \to \im \eta_{i+1} \hookrightarrow Y_{i+2}&= X_{i+1} \to \im \xi_{i+1} \to X_{i+2} \to Y_{i+2}\\& = X_{i+1} \to Y_{i+1} \to Y_{i+2}\\& = X_{i+1} \to Y_{i+1} \to \im \eta_{i+1} \hookrightarrow Y_{i+2},
\end{split}
\end{equation*}
so $X_{i+1} \to \im \xi_{i+1} \to \im \eta_{i+1}=X_{i+1} \to Y_{i+1} \to \im \eta_{i+1}$. Since $\varphi_{i+1}$ is the only one making the middle square commutative by the universal property of the cokernel, TR3 entails that also the right-hand square is commutative.

We obtain a commutative diagram 
\[
\begin{tikzcd}
A\ar[d,"\varphi_n"] \ar[r]&\im \xi_{n-1}[1] \arrow{d}{\varphi_{n-1}[1]} \ar[r]& \dots \ar[r]& \im \xi_{2}[n-2] \arrow{d}{\varphi_{2}[n-2]}\ar[r]& \im \xi_{1}[n-1] \arrow{d}{\varphi_{1}[n-1]}\ar[r] &B[n]\arrow{d}{\varphi_0 [n]}\\
A\ar[r]&\im \eta_{n-1}[1] \ar[r]& \dots \ar[r]& \im \eta_{2}[n-2] \ar[r]& \im \eta_{1}[n-1] \ar[r] &B[n]
\end{tikzcd}
\]
where $\varphi_n=\id$ and $\varphi_0=\id$, so that the rows are in fact the same map. This gives the well-definition of every $f_{n,A,B}$. 
\begin{enumerate}
\item Let us consider a map $\alpha:A \to B[n]$ factoring through $A=C_n \to C_{n-1}[1] \to \dots \to C_1[n-1]\to C_0[n]=B[n]$. To any $C_{i}[-1] \to C_{i-1}$, we can associate a cone, which is in $\mathcal{A}$ by Theorem \ref{thm:dyer}. Let us call such cone $X_i$. We have the following short exact sequences: $0 \to C_{i-1} \to X_i \to C_i\to 0$. Since $C_{i}$ is also the kernel of $X_{i+1} \to C_{i+1}$, we manage to create  an exact sequence
\[
0\to B \to X_1 \to X_2 \to \dots \to X_n \to A \to 0.
\]
It is easy to notice that such exact sequence is associated to the map $\alpha: A\to B[n]$ via $f_{n,A,B}$.
\item In the case of $\ext^n$ and $\ext^m$ with $n,m>0$, the Yoneda product is sent to composition with a reasoning similar to item 1. Therefore, it suffices to show it is true when either $m$ or $n$ is zero.  
First, we recall that $f_{1,A,B}$ is exactly the map considered in Theorem \ref{thm:dyer}, which is a natural transformation for both entries. So \eqref{eq:Xg} can be translated to
\begin{equation}\label{eq:Xg1}
\begin{tikzcd}
B \ar[r] \ar[d,"g"] & X_1 \ar[r]\ar[d]&K\ar[r,"h"]\ar[d,"\id"] & B[1]\arrow{d}{g[1]}\\
C \ar[r] & g \cdot X_1 \ar[r] & K \arrow{r}{g[1] f} & C[1]
\end{tikzcd}
\end{equation} in $\T$. Let us prove that $f_{n,A,-}$ is a natural transformation, the proof of $f_{n,-,B}$ being dual.
For a general $n$-extension 
\[
\mathbf{X}:0 \to B \to X_1 \to X_2 \to \dots \to X_n \to A \to 0
\]
and $g:B \to C$, the map $A \to C[n]$ associated to $g \cdot \mathbf{X}$ factors through $K[n-1]\to C[n]$, where $K= \im (g\cdot X_1 \to X_2)= \im(X_1\to X_2)$, according to \eqref{eq:Xg1}. Furthermore, the same diagram shows that $K \to C[1]$ is obtained as a composition $K \to B[1] \to C[1]$, where the latter morphism is $g[1]$. Therefore, $A \to C[n]$ can be written as the composition of $A \to B[n]$, obtained by $\mathbf{X}$, and $g[n]: B[n] \to C[n]$, as wanted.


Finally, $f_{n,-,-}$ is a natural transformation for both entries $A$ and $B$. Moreover, $f_{n,A,B}$ is a group homomorphism since the Baer sum of two extensions is given by Yoneda products as explained in Definition \ref{defn:extn}.
\item We want to show that the zero map $A \to B[n]$ is associated to only one equivalence class of extensions, the trivial one, whenever $f_{n-1,A,X}$ is an isomorphism for any $X \in \mathcal{A}$. 

Let us consider 
\[
\mathbf{X}:0 \to B \to X_1 \to X_2 \to \dots \to X_n \to A \to 0\] 
such that $f_{n,A,B}(\mathbf{X})=0$ and the associated factorization
\[
A \to C_{n-1}[1] \to \dots \to C_2[n-2] \to C_1[n-1] \to B[n].
\]
We have the following diagram, where the rows are distinguished triangles:
\begin{equation}
\label{eq:x1}
\begin{tikzcd}
A\ar[r,"0"]\ar[d] & B[n]\ar[r]\ar[d,"\id"] & B[n]\oplus A[1] \ar[r] \ar[d] & A[1]\ar[d]\\
C_1[n-1] \ar[r]&B[n]\ar[r]&X_1[n]\arrow{r}{g[n]}&C_1[n]
\end{tikzcd}
\end{equation}
Now we pick the map $A[1] \to B[n] \oplus A[1] \to X_1[n]$. Since $f_{n-1,A,X_1}$ is a surjective, we get that $A\to X_1[n-1]$ is associated to an exact sequence
\[
\mathbf{Y}: 0 \to X_1 \to Y_1 \to \dots \to Y_{n-1} \to A \to 0.
\]
Composing $\mathbf{Y}$ with $0 \to B \to X_1 \oplus B \to X_1 \to 0$, we have the following:
\begin{equation}
\label{eq:xy}
\begin{tikzcd}
0 \ar[r]&B\arrow{r}{\left(\begin{smallmatrix}0\\ \id \end{smallmatrix}\right)}\ar[d,"\id"]& X_1 \oplus B \ar[r]\arrow{d}{(\id, \iota)}& Y_1 \ar[r] & \dots \ar[r]& Y_{n-1}\ar[r]&A\ar[r]\ar[d,"\id"] & 0\\
0 \ar[r]&B\ar[r,"\iota"]&X_1\ar[r]&X_2 \ar[r]&\dots \ar[r]&X_n \ar[r]&A\ar[r]&0.
\end{tikzcd}
\end{equation}
We want to prove there are maps $Y_i \to X_{i+1}$ making every square of the diagram above commutative. It suffices to consider the sequences starting at $X_1$ and $C_1$ respectively (remember that $C_1$ is the image of $X_1 \to X_2$). The Yoneda product of $\mathbf{Y}$ and $g:X_1 \to C_1$ gives us $g \cdot \mathbf{Y}$, whose associated map $A\to X_1[n-1] \to C_1[n-1]$ factors as $A \to C_{n-1}[1] \to \dots \to C_1[n-1]$ because of the right-hand commutative square in \eqref{eq:x1}. Since $f_{n-1,A,C_1}$ is injective by assumption, we know that $g \cdot \mathbf{Y}$ is in the same equivalence class of
\[
0 \to C_1 \to X_2 \to \dots \to X_n \to A \to 0 .
\]
Therefore, we can assume, up to equivalence, that $\mathbf{X}$ is in fact
\[
0 \to B \to X_1 \to g \cdot Y_1 \to Y_2 \to \dots \to Y_{n-1}\to A \to 0.
\]
With this assumption, \eqref{eq:xy} can be completed with maps $Y_i \to X_{i+1}$ as wanted: the first morphism is given according to \eqref{eq:Xg}, while all the others are the identity. It remains to show that the equivalence class of
\[
0 \to B \to X_1 \oplus B \to Y_1 \to \dots \to Y_{n-1} \to A\to 0
\]
is the one associated to 0, which is obvious because the diagram
\[
\begin{tikzcd}[column sep=scriptsize]
0 \ar[r]&B\ar[d,"\id"]\arrow{r}{\left(\begin{smallmatrix}0\\ \id \end{smallmatrix}\right)}& X_1 \oplus B \ar[r]\ar[d,"(0\ \id)"]& Y_1 \ar[r]\ar[d] & \dots \ar[r]&Y_{n-2}\ar[r]\ar[d] &Y_{n-1}\ar[r,"\pi"]\ar[d,"\pi"]&A\ar[r]\ar[d,"\id"] & 0\\
0 \ar[r]&B\ar[r,"\id"]&B\ar[r]&0 \ar[r]&\dots \ar[r]&0\ar[r]&A \ar[r,"\id"]&A\ar[r]&0
\end{tikzcd}
\]
commutes.
\item Let $g_{n,A,B}$ as in the statement and assume by induction that $g_{m,C,D}=f_{m,C,D}$ for any $m<n$ and $C,D \in \mathcal{A}$. We consider $\mathbf{X} \in \ext^n(A,B)$ given by 
\[
0 \to B \to X_1 \overset{\xi_1}{\to} X_2 \to \dots \to X_n \to A \to 0.\]
Such an extension can be split into two shorter extensions: 
\begin{equation*}
\begin{split}
\mathbf{X}_1:&\quad 0 \to B \to X_1 \to \coker (\xi_1) \to 0\\
\mathbf{X}_2:&\quad 0 \to \coker (\xi_1) \to X_2 \to \dots \to X_n \to A \to 0.
\end{split}
\end{equation*}
Moreover, $\mathbf{X}_1 \cdot \mathbf{X}_2=\mathbf{X}$. As $g_{n,A,B}$ sends Yoneda product to composition, we have
\begin{equation*}
\begin{split}
g_{n,A,B}(\mathbf{X})&= g_{n,A,B}(\mathbf{X}_1 \cdot \mathbf{X}_2) \\
&= g_{1,\coker(\xi_1),B}(\mathbf{X}_1) \circ g_{n-1,A,\coker(\xi_1)}(\mathbf{X}_2)\\
&= f_{1,\coker(\xi_1),B}(\mathbf{X}_1) \circ f_{n-1,A,\coker(\xi_1)}(\mathbf{X}_2)\\
&= f_{n,A,B}(\mathbf{X}_1 \cdot \mathbf{X}_2)= f_{n,A,B}(\mathbf{X}).
\end{split}
\end{equation*}
\end{enumerate}
\end{proof}
\begin{rem}\label{rem:heart1}
By Proposition \ref{prop:yonext}, for any exact subcategory $\mathcal{A} \subset \T$ as in Dyer's Theorem \ref{thm:dyer}, it holds that $\ext^2(A,B) \subset \hom(A,B[2])$ for any $A,B \in \mathcal{A}$.
In case $\mathcal{A}$ is a heart, $\dim_{\mathcal{T}}(\mathcal{A})\le 1$ implies that $\dim \mathcal{A}\le 1$.
\end{rem}
\begin{defn}\label{def:allext}
Let $\mathcal{T}$ be a triangulated category and $\mathcal{A}$ an exact subcategory as in Dyer's Theorem \ref{thm:dyer}.
We say that $\mathcal{T}$ \emph{has all the Ext groups of $\mathcal{A}$} if the morphism $f_{n,A,B}$ defined in Proposition \ref{prop:yonext} is an isomorphism for any $A,B \in \mathcal{A}$ and all $n \in \mathbb{N}$.
\end{defn}

\begin{cor}\label{cor:extgroups}
A triangulated category $\mathcal{T}$ has all the Ext groups of an exact subcategory $\mathcal{A}$ as in Dyer's Theorem \ref{thm:dyer} if and only if for every map $A \to B[n]$ there exists a factorization
\[
A \to C_{n-1}[1] \to \dots \to C_1[n-1] \to B[n]
\]
with $C_i \in \mathcal{A}$ for $i \in \lbrace 1, \dots , n-1\rbrace $.
In particular, if $\mathcal{A}$ is a heart and $\dim_{\T}\mathcal{A}\le 1$, then $\T$ has all the Ext groups of $\mathcal{A}$ and $\dim\mathcal{A}=\dim_{\T}\mathcal{A}$.
\end{cor}
\begin{proof}
The only if part is obvious: if $f_{n,A,B}$ is an isomorphism, then the image of such map contains all morphisms $A \to B[n]$: item 1 of Proposition \ref{prop:yonext} concludes. 

Conversely, item 1 of Proposition \ref{prop:yonext} shows that $f_{n,A,B}$ is surjective. By Theorem \ref{thm:dyer}, $f_{1,A,B}$ is an isomorphism: we obtain that $f_{2,A,B}$ is injective according to item 3 of Proposition \ref{prop:yonext}. An induction proves that this holds for every $n$.

Using Remark \ref{rem:heart1} and Theorem \ref{thm:dyer}, we prove the last part of the statement.
\end{proof}
\begin{cor}\label{cor:allext}
Let $\T$ be a triangulated category with a semiorthogonal decomposition $\T=\langle \T_1 , \T_2 \rangle$ and two compatible t-structures $\T_1^{\le 0}$ and $\T_2^{\le 0}$ on $\T_1$ and $\T_2$ respectively. We denote with $\mathcal{A}_i$ the heart associated to $\T_i^{\le 0}$. By Theorem \ref{thm:decheart}, we obtain the heart 
\[\mathcal{A}= \mathcal{A}_2 *\mathcal{A}_1[1].
\]
We consider the following hypotheses:
\begin{enumerate}
\item $\T_i$ has all the Ext groups of $\mathcal{A}_i$;
\item $\hom_{\T}(A,B[m])\cong \ext_{\mathcal{A}}^m(A,B)$ for every $A \in \mathcal{A}_1[1]$ and $B \in \mathcal{A}_2$.
\end{enumerate}
Then $\T$ has all the Ext groups of the heart $\mathcal{A}$.
\end{cor}
\begin{proof}
Before starting the actual proof, let us remark that $\ext^m_{\mathcal{A}}(A,B) = \ext^m_{\mathcal{A}_2} (A,B)$ whenever $A,B \in \mathcal{A}_2$. Indeed, let 
\[
\mathbf{X}: 0 \to B \to X_1 \to X_2 \to \dots \to X_n \to A \to 0
\]
be an extension in $\mathcal{A}$ with $A,B \in \mathcal{A}_2$ and let $\sigma_2 : \T \to \T_2$ be the right adjoint of the inclusion functor $\iota: \T_2 \to \T$. Then we get 
\[
\begin{tikzcd}
\iota\sigma_2 \mathbf{X}: & 0 \ar[r]&B\ar[d,"\id"]\ar[r] & \iota \sigma_2 X_1 \ar[d]\ar[r] & \dots\ar[r] & \iota \sigma_2 X_n \ar[d]\ar[r] & A \ar[d,"\id"]\ar[r] & 0\\ 
\mathbf{X}:&0 \ar[r]&B\ar[r] & X_1 \ar[r] & \dots \ar[r]& X_n \ar[r] & A \ar[r] & 0 
\end{tikzcd}
\]
which shows that $\iota \sigma_2 \mathbf{X} \cong \mathbf{X}$ in $\ext^m_{\mathcal{A}}(A,B)$ (recall the equivalence relation used to describe the Yoneda extensions in Definition \ref{defn:extn}). Since $\sigma_2\mathbf{X} \in \ext^m_{\mathcal{A}_2}(A,B)$, we conclude that $\iota$ gives an isomorphism between $\ext^m_{\mathcal{A}_2}(A,B)$ and $\ext^m_{\mathcal{A}}(A,B) $ whenever $A,B \in \mathcal{A}_2$. In a similar way, considering the left adjoint of the inclusion $\T_1 \to \T$, one can prove that $\ext^m_{\mathcal{A}}(A,B)= \ext^m_{\mathcal{A}_1[1]}(A,B)$ if $A,B \in \mathcal{A}_1[1]$.

Given $A,B \in \mathcal{A}$, we consider two distinguished triangle $A_2 \to A \to A_1 \to A_2[1]$ and $B_2 \to B \to B_1 \to B_2[1]$ with $A_2,B_2 \in \mathcal{A}_2$ and $A_1 , B_1 \in \mathcal{A}_1[1]$. We obtain the following hom-exact sequences
\[
\begin{tikzcd}[row sep=tiny,column sep=1.7em]
\dots \ar[r]& \hom(A_1,B[m]) \ar[r]& \hom(A,B[m]) \ar[r]& \hom (A_2,B[m]) \ar[r]& \dots\\
\dots \ar[r]& \hom(A_1,B_2[m]) \ar[r]& \hom(A_1,B[m]) \ar[r]& \hom(A_1 , B_1[m]) \ar[r]& \dots\\
\dots \ar[r]& \hom(A_2, B_2[m]) \ar[r]& \hom(A_2,B[m]) \ar[r]& \hom(A_2 , B_1[m])=0 \ar[r]& \dots
\end{tikzcd}
\]
By Proposition \ref{prop:yonext}, these exact sequences have maps from the Ext groups. We proceed by induction on $m$. From the induction hypothesis and item 3 of Proposition \ref{prop:yonext} we deduce that \[\ext^m_{\mathcal{A}}(A_2,B_1) \subseteq \hom(A_2,B_1[m])=0.\] Therefore, hypothesis 1 and five lemma entails that $\hom(A_2,B[m])\cong \ext_{\mathcal{A}}^m(A_2,B)$. The second row proves that $\hom(A_1,B[m])\cong \ext_{\mathcal{A}}^m(A_1,B)$ using both hypotheses and five lemma. From the first row, we conclude that $\hom(A,B[m])\cong \ext_{\mathcal{A}}^m(A,B)$.
\end{proof}

\begin{prop}\label{prop:der}\emph{\cite[Propositions XI.4.7 and 4.8]{iversen}.}
In the case of $\der(\mathcal{A})$, the map $f_{n,A,B}:\ext^n(A,B) \to \hom(A,B[n])$ above is exactly the classical one, that associates to each $\mathbf{X} \in \ext^n(A,B)$ the map given by the composition of the inverse of the quasi-isomorphism
\[
(0 \to B \to X_1 \to \dots \to X_n \to 0) \to A
\]
(the left-hand complex is such that $X_n$ is at level 0) and the morphism
\[
(0 \to B \to X_1 \to \dots \to X_n \to 0) \to B[n].
\]
In particular, in the case of $\der(\mathcal{A})$ every $f_{n,A,B}$ is an isomorphism.
\end{prop}
\begin{proof}
This is a direct consequence of item 4 of Proposition \ref{prop:yonext}. 
The last sentence is a classical result; see, for instance, \cite[Proposition XI.4.8]{iversen}.
\end{proof}
\noindent {\large \textbf{Data Availability Statement}}\hspace{2ex} NA

\phantomsection
\addcontentsline{toc}{section}{\refname}
\printbibliography

\end{document}

%% file: structure-article.tex
\usepackage[utf8]{inputenc}
\usepackage[english]{babel}
\usepackage[T1]{fontenc}

\usepackage{mathptmx}
\DeclareMathAlphabet{\mathscr}{OMS}{cmsy}{m}{n} 
\usepackage{geometry} 
\usepackage{enumitem} 
\setlist{nolistsep} 

\usepackage{booktabs} 
\usepackage{array} 
\usepackage[labelfont=sc]{caption}

\title{Compatibility of t-structures in a semiorthogonal decomposition}
\usepackage{fancyhdr} 
\makeatletter
\let\runtitle\@title
\makeatother
\usepackage{titletoc}
\titlecontents{section}[0em]{}{\thecontentslabel{.\hspace{6pt}}}{\hspace*{0em}}{\titlerule*[0.5pc]{.}\contentspage}

\usepackage[outermarks]{titlesec}

\titleformat{\section}
{\bfseries\Large}
{\thesection.}
{1ex}
{}
[]
\titleformat{\subsection}
{\normalfont
\titlerule*[.2em]{--}%
\vspace{4pt}%
\bfseries\footnotesize}
{}{0em}{\MakeUppercase}

\usepackage{tikz}
\usetikzlibrary{shapes,cd,decorations.pathmorphing}
\usepackage{amssymb}
\usepackage{amsmath}

\usepackage{calc}
\usepackage{mathtools}
\DeclarePairedDelimiter\abs{\lvert}{\rvert}
\DeclarePairedDelimiter\norm{\lVert}{\rVert}
\makeatletter
\let\oldabs\abs
\def\abs{\@ifstar{\oldabs}{\oldabs*}}
\let\oldnorm\norm
\def\norm{\@ifstar{\oldnorm}{\oldnorm*}}
\makeatother

\usepackage{amsthm}
\newtheoremstyle{plainmod}
{}
{}
{\it}
{}
{}
{}
{ }
{\textbf{\thmnumber{#2}.\thmname{\hspace{1ex}#1}}\thmnote{\textbf{ -- #3}}\textbf{.}}
\newtheoremstyle{definitionmod}
{}
{}
{}
{}
{}
{}
{ }
{\textbf{\thmnumber{#2}.\thmname{\hspace{1ex}#1}}\thmnote{\textbf{ -- #3}}\textbf{.}}
\newtheoremstyle{remarkmod}
{}
{}
{}
{}
{}
{}
{ }
{\textsc{\thmnumber{#2}.\thmname{\hspace{1ex}#1}}\thmnote{\textsc{ -- #3}}\textit{.}}

\numberwithin{equation}{section}

\theoremstyle{plainmod} 
\newtheorem{thm}[equation]{Theorem}
\newtheorem{cor}[equation]{Corollary}
\newtheorem{lem}[equation]{Lemma}
\newtheorem{prop}[equation]{Proposition}

\theoremstyle{definitionmod} 
\newtheorem{defn}[equation]{Definition}
\newtheorem{es}[equation]{Example}
\newtheorem{rem}[equation]{Remark}

\newtheorem{nota}[equation]{Notation}

\newtheorem{defnp}[equation]{Definition/Proposition}
\expandafter\let\expandafter\oldproof\csname\string\proof\endcsname
\let\oldendproof\endproof
\renewenvironment{proof}[1][\proofname]{%
  \oldproof[\scshape #1.]%
}{\oldendproof}